\documentclass[12pt]{amsart}
\usepackage{amsfonts}
\usepackage{amssymb,amsthm,amscd,amstext,amsmath,mathrsfs,latexsym}
\usepackage{tikz}
\usetikzlibrary{intersections,decorations.pathmorphing,decorations.fractals}
\newtheorem{theorem}{Theorem}
\newtheorem{lemma}{Lemma}

\newtheorem{cor}{Corollary}
\newtheorem{prop}{Proposition}
\newtheorem{remark}{Remark}
\newtheorem{example}{Example}
\newtheorem{definition}{Definition}

\begin{document}

\title[]
{Kitaev models based on unitary quantum groupoids}

\author{Liang Chang}
\address{Department of Mathematics\\
    Texas A\&M University \\
    College Station, TX 77843-3368}
\email{liangchang@math.tamu.edu}

\begin{abstract}
We establish a generalization of Kitaev models based on unitary quantum groupoids. In particular, when inputting a Kitaev-Kong quantum groupoid $H_\mathcal{C}$, we show that the ground state manifold of the generalized model is canonically isomorphic to that of the Levin-Wen model based on a unitary fusion category $\mathcal{C}$. Therefore the generalized Kitaev models provide realizations of the target space of the Turaev-Viro TQFT based on $\mathcal{C}$.
\end{abstract}

\maketitle

\section{Introduction}
In \cite{Kit1}, Kitaev proposed an approach to quantum computation which is based on quantum many-body systems exhibiting topological order, i.e., systems that are effectively described by a Topological quantum field theory (TQFT). Given a finite group $G$, Kitaev constructed a Hilbert space on a triangulated surface and an exactly solvable Hamiltonian, whose ground state is a topological invariant of the surface. The best known of these models is the toric code, which is based on $\mathbb{Z}_2$. Recently, the semisimple Hopf algebra extension was achieved in \cite{BCMA}. Then in \cite{BK2} it was proved that the ground state manifold of the Kitaev model based on a $C^*$-Hopf algebra $H$ is canonical isomorphic to the ground state manifold of the Levin-Wen model based on the category $Rep(H)$.

In condensed matter physics, TQFTs are used to describe topological phases of matter. Turaev-Viro TQFTs (TV-TQFTs) are realized by Levin-Wen models which was introduced in \cite{LW} and understood rigorously in \cite{KK}. Given a unitary fusion category $\mathcal{C}$ and a trivalent lattice $\Gamma$ on a surface $\Sigma$, one can write down a local exactly solvable Hamiltonian and the space of ground states turns out to be canonical isomorphic to the target vector space $Z_{TV}(\Sigma)$ of the TV-TQFT based on $\mathcal{C}$ \cite{Kir}.

It is known that \cite{ENO} every unitary fusion category is the representation category of a $C^*$-quantum groupoid, which is not a Hopf algebra in general. Therefore the full dual Kitaev models to Levin-Wen Models should take $C^*$-quantum groupoids as inputs. It was conjectured \cite{BCKA} that the construction in  \cite{BCMA} should work for $C^*$-quantum groupoid. In this paper, we establish this construction for certain $C^*$-quantum groupoids and obtain a similar isomorphism between the ground state manifolds of generalized Kitaev models and Levin-Wen Models.

In \cite{KK}, a $C^*$-quantum groupoid $H_\mathcal{C}$ was defined from a unitary fusion category $\mathcal{C}$ to study the boundary excitations of Levin-Wen models. In this paper, Kitaev models are generalized based on $H_\mathcal{C}$. To this end, on a given lattice of a closed oriented surface, vertex operators $A_\Lambda(\mathbf{v})$'s and plaquette operators $B_\lambda(\mathbf{p})$ are defined using suitable cocommutative elements $\Lambda$ and $\lambda$ in $H_\mathcal{C}$ and $H_\mathcal{C}^*$ so that they commute with each other and
  $$\mathcal{H}^K=-\sum\limits_{\mathbf{v}}A_\Lambda(\mathbf{v})-\sum\limits_{\mathbf{p}}B_\lambda(\mathbf{p})$$
is a frustration-free Hamiltonian (see more details in section 5 and section 7).

The relationship between the work of Levin-Wen and Kitaev was discussed in \cite{BA} and \cite{KMR}. For mathematicians, \cite{BK2} provides a readable explanation in the case of semisimple Hopf algebras. In this case, one can find a 1-1 correspondence between ground states of these two models. For the $C^*$-quantum groupoids $H_\mathcal{C}$, we verify that such one-one correspondence still holds. More precisely, under certain assumption, given a unitary fusion category $\mathcal{C}$ and a lattice $\Gamma$ on a closed oriented surface $\Sigma$, the ground state space $\mathcal{G}^{K}(\Sigma,\Gamma)$ of the generalized Kitaev model based on $H_\mathcal{C}$ is canonically isomorphic to the ground state space $\mathcal{G}^{LW}(\Sigma,\Gamma)$ of Levin-Wen Models based on $\mathcal{C}$. As a consequence, $\mathcal{G}^{K}(\Sigma,\Gamma)$ is canonically isomorphic to the target space $Z_{TV}(\Sigma)$ of the TV-TQFT based on $\mathcal{C}$.

The contents of the paper are as follows. In section 2 and section 3, we recall the notions of unitary fusion category and the construction of Levin-Wen models. In section 4 and section 5, we recall the notion of $C^*$-quantum groupoid and construct the generalized Kitaev models based on $C^*$-quantum groupoids. In section 6, we set up the notion of Kitaev-Kong $C^*$-quantum groupoid $H_\mathcal{C}$ based on a unitary fusion category $\mathcal{C}$ and study its representation categories. Finally, in section 7, we input $H_\mathcal{C}$ to write down the Kitaev model and compares its ground states with those of LW models based on $\mathcal{C}$.\\

\noindent
\textbf{Acknowledgements}. The author is grateful to Zhenghan Wang for his advice and encouragement.

\section{Unitary fusion categories}

A fusion category is a semisimple abelian rigid tensor category with finitely many isomorphism classes of simple objects and finite dimensional morphism spaces and simple unit object (see \cite{BK} or \cite{ENO} for the complete axiomatic definition). In this section we recall the notion of unitary fusion category and establish our notations of certain $6j$-symbols.

From now on, $Irr(\mathcal{C})$ denotes the set of isomorphism classes of non-zero simple objects in $\mathcal{C}$. The set of decompositions $i\otimes j\cong \bigoplus_{k\in Irr(\mathcal{C})}N^k_{ij}k$ for all $i,j\in Irr(\mathcal{C})$ and some natural number $N^k_{ij}$, is called the fusion rule of the fusion category $\mathcal{C}$. We say a triple of simple objects $(i,j,k)$ is admissible if $N^k_{ij}\neq 0$. A fusion category is multiplicity-free if $N^k_{ij}\in \{0,1\}$ for any simple objects $(i,j,k)$, i.e., $dim(Hom(i\otimes j,k))=0~\text{or}~1$.

In the following, we employ graphical calculus that complies the conventions in \cite{BK}. In particular, diagrams are read bottom to top. 

If $a$ is a self-dual simple object in a pivotal fusion category, i.e.,
$a^*=a$, then $dim~Hom(a\otimes a, \textbf{1})=dim~Hom(a^*,
a)=dim~Hom(a, a)=1$. This implies that as vectors in $Hom(a\otimes
a, \text{1})$,
\[
\begin{tikzpicture}[scale=0.5]
  \begin{scope}[xshift=-2.3cm]
    \draw[line width=0.5mm] (1,0) arc(0:180:1);
    \node (a) at (1,-0.3) {$a$};
    \node (b) at (-1,-0.3) {$a$};
  \end{scope}
  \node (=) at (0,0.3) {$=\nu_a$};
  \begin{scope}[xshift=4cm]
    \draw[line width=0.5mm] (-1,0) arc(0:180:0.5);
    \draw[line width=0.5mm] (-2,0) arc(360:180:0.5 and 0.4);
    \draw[line width=0.5mm] (-3,0) arc(180:0:2 and 1.5);
    \node (a) at (1,-0.3) {$a$};
    \node (b) at (-1,-0.3) {$a$};
  \end{scope}
\end{tikzpicture}
\]
Such scalar $\nu_a$ is called the Frobenius-Schur indicator. It turns out that $\nu_a=\pm1$ \cite{Wangbook}.

For simplicity we will only consider multiplicity-free fusion categories whose simple objects are self-dual and  Frobenius-Schur indicators are trivial, i.e. $\nu_V=1$ for every simple object. This assumption allows removing arrows from graphs and not considering the $\pm$ sign from Frobenius-Schur indicators when wrapping the lines. 

The associativity of a fusion category can be represented by a family of numbers. They are the so called $6j$-symbols. For $a,b,c,d\in Irr(\mathcal{C})$, we choose bases for the vector spaces $Hom((a\otimes b)\otimes c,d)$ and $Hom(a\otimes(b\otimes c),d)$ which are presented by the trivalent graphs
\[
\begin{tikzpicture}[scale=0.5]    
  \begin{scope}[xshift=-8cm]
    \draw[line width=0.5mm] (0,0)--(-2,2);
    \draw[line width=0.5mm] (0,0)--(2,2);
    \draw[line width=0.5mm] (-1,1)--(0,2);
    \draw[line width=0.5mm] (0,0)--(0,-1);
    \node (a) at (-2,2.4) {$a$};
    \node (b) at (0,2.5) {$b$};
    \node (c) at (2,2.4) {$c$};
    \node (d) at (0,-1.4) {$d$};
    \node (m) at (-1,0.3) {$m$};
    \node (Hom) at (6.5,0.5) {$\in Hom((a\otimes b)\otimes c,d)$,};
  \end{scope}
  \begin{scope}[xshift=8cm]
    \draw[line width=0.5mm] (0,0)--(-2,2);
    \draw[line width=0.5mm] (0,0)--(2,2);
    \draw[line width=0.5mm] (1,1)--(0,2);
    \draw[line width=0.5mm] (0,0)--(0,-1);
    \node (a) at (-2,2.4) {$a$};
    \node (b) at (0,2.5) {$b$};
    \node (c) at (2,2.4) {$c$};
    \node (d) at (0,-1.4) {$d$};
    \node (n) at (1,0.3) {$n$};
    \node (Hom) at (6.5,0.5) {$\in Hom(a\otimes (b\otimes c),d)$};
  \end{scope}
\end{tikzpicture}
\]
where $m$ and $n$ run through all possible admissible simple objects. The $F$-matrices, whose entries are $6j$-symbols, present the isomorphism $F^{abc}_d:Hom((a\otimes b)\otimes c, d)\rightarrow Hom(a\otimes(b\otimes c),d)$.
\[
\begin{tikzpicture}[scale=0.5]    
  \begin{scope}[xshift=-4.5cm]
    \draw[line width=0.5mm] (0,0)--(-2,2);
    \draw[line width=0.5mm] (0,0)--(2,2);
    \draw[line width=0.5mm] (-1,1)--(0,2);
    \draw[line width=0.5mm] (0,0)--(0,-1);
    \node (a) at (-2,2.4) {$a$};
    \node (b) at (0,2.5) {$b$};
    \node (c) at (2,2.4) {$c$};
    \node (d) at (0,-1.4) {$d$};
    \node (m) at (-1,0.3) {$m$};
  \end{scope}
  \node (=) at (0,0.3) {$=\sum\limits_nF^{abc}_{d;nm}$};
  \begin{scope}[xshift=4.3cm]
    \draw[line width=0.5mm] (0,0)--(-2,2);
    \draw[line width=0.5mm] (0,0)--(2,2);
    \draw[line width=0.5mm] (1,1)--(0,2);
    \draw[line width=0.5mm] (0,0)--(0,-1);
    \node (a) at (-2,2.4) {$a$};
    \node (b) at (0,2.5) {$b$};
    \node (c) at (2,2.4) {$c$};
    \node (d) at (0,-1.4) {$d$};
    \node (n) at (1,0.3) {$n$};
  \end{scope}
\end{tikzpicture}
\]

The associativity of $\mathcal{C}$ can be translated into the Pentagon equations among $F^{abc}_{d;nm}$'s: for all $a, b,c,d,e,f,p,q,m\in Irr(\mathcal{C})$,
\begin{equation}\label{eqn:pentagon}
\sum\limits_nF^{bcd}_{q;pn}F^{and}_{f;qe}F^{abc}_{e;nm}=F^{abp}_{f;qm}F^{mcd}_{f;pe}
\end{equation}

\begin{definition}
  A unitary fusion category $\mathcal{C}$ is a fusion category over $\mathbb{C}$ equipped with a conjugation $Hom_{\mathcal{C}}(U,V)\rightarrow Hom_{\mathcal{C}}(V,U)$ denoted by $f\mapsto\bar{f}$ for all objects $U,V$ of $\mathcal{C}$ satisfying:\\
  (i) $\overline{\lambda f}=\bar{\lambda}\bar{f}$, for $\lambda\in \mathbb{C}$.\\
  (ii) $\overline{g\circ f}=\bar{f}\bar{g}$ for $f\in Hom_{\mathcal{C}}(U,V)$ and $g\in Hom_{\mathcal{C}}(V,W)$.\\
  (iii) $\overline{f\otimes g}=\bar{f}\otimes\bar{g}$. for $f\in Hom_{\mathcal{C}}(U,V)$ and $g\in Hom_{\mathcal{C}}(U',V')$.\\
  (iv) $\bar{\bar{f}}=f$.\\
  (v) $\bar{f}\circ f=0$ if and only if $f=0$.
\end{definition}

\begin{prop}
  \cite{Kit2} $F$-matrices of a unitary fusion category can be chosen to be unitary.
\end{prop}
Indeed, one can choose $F^{abc}_d$ such that $F^{abc}_d(F^{bcd}_a)^\dag=I$. Moreover, because of the physical application the theta symbols are normalized as
\[
\begin{tikzpicture}[scale=0.5]
  \node (theta) at (0,0) {$\theta(a,b,c):=$};
  \begin{scope}[xshift=4.5cm]
    \draw[line width=0.5mm] (2,0) arc (0:360:2cm and 1cm);
    \draw[line width=0.5mm] (-2,0)--(2,0);
    \node (a) at (0,1.4) {$a$};
    \node (b) at (0,0.4) {$b$};
    \node (c) at (0,-1.4) {$c$};
  \end{scope}
  \node (=) at (9,0) {$=\sqrt{d_ad_bd_c}$};
\end{tikzpicture}
\]
As a result, the identity homomorphism from $a\otimes b$ to itself can be decomposed as follows:
\begin{equation}\label{eqn:Fmove}
\begin{split}
\begin{tikzpicture}[scale=0.5]
  \begin{scope}[xshift=-4cm]
  \draw[line width=0.5mm] (-1,1)--(-1,-1);
  \draw[line width=0.5mm] (1,1)--(1,-1);
  \node (a) at (-1.3,1.3) {$$};
  \node (b) at (-1.3,-1.3) {$a$};
  \node (c) at (1.3,1.3) {$$};
  \node (d) at (1.3,-1.3) {$b$};
  \node (i) at (0,0.5) {$$};
  \end{scope}
  \node (=) at (0,-0.2) {$=\sum\limits_{n}\sqrt{\frac{d_n}{d_ad_b}}$};
  \begin{scope}[xshift=4cm]
  \draw[line width=0.5mm] (-1,1)--(0,0.7);
  \draw[line width=0.5mm] (1,1)--(0,0.7);
  \draw[line width=0.5mm] (-1,-1)--(0,-0.7);
  \draw[line width=0.5mm] (1,-1)--(0,-0.7);
  \draw[line width=0.5mm] (0,-0.7)--(0,0.7);
  \node (a) at (-1.3,1.3) {$a$};
  \node (b) at (-1.3,-1.3) {$a$};
  \node (c) at (1.3,1.3) {$b$};
  \node (d) at (1.3,-1.3) {$b$};
  \node (i) at (0.5,0) {$n$};
  \end{scope}
\end{tikzpicture}
\end{split}
\end{equation}

\section{Levin-Wen models based on unitary fusion categories}

Let $\mathcal{C}$ be a unitary fusion category. Given a trivalent lattice $\Gamma$ on an oriented closed surface $\Sigma$,
the Hilbert space for this model is
$$\mathcal{L}^{LW}=\bigotimes_{\text{edges}}\mathbb{C}^l$$
where $l$ is the rank of $\mathcal{C}$. It has a natural basis of all edge-labelings of $\Gamma$ by representatives of simple objects (called labels).
For simplicity, we assume $\mathcal{C}$ is multiplicity-free and self dual. The Hamiltonian will be written as
$$\mathcal{H}^{LW}=-\sum\limits_{\mathbf{v}}A^{LW}_{\mathbf{v}}-\sum\limits_{\mathbf{v}}B^{LW}_{\mathbf{p}}$$
It suffices to define these operators on each basis vector. Given an edge-labeling $e_l$ of $\Gamma$, define $A^{LW}_{\mathbf{v}}|e_l\rangle=|e_l\rangle$
if the three labels are admissible around $\mathbf{v}$, else $A^{LW}_{\mathbf{v}}|e_l\rangle=0$. Then the subspace $\mathcal{L}^{LW}_0=A^{LW}_{\mathbf{v}}(\mathcal{L}^{LW})$
is spanned by the edge-labelings of which any three labels around each vertex are admissible. The plaquette term will be expected to satisfy the requirement of zero total flux through each plaquette.
An explicit formula is given by
\[

\]
where $D^2=\sum\limits_kd_k^2$ and the effect of $B^{\mu}_\mathbf{p}$ is to impose a circle labeled by $\mu$ into the plaqutte.
\[
%
\]

Here the $\mu$-circle can be removed by a sequence of $F$-moves \eqref{eqn:Fmove}. Then a bubble is created at each vertex.
\[
%
\]

Actually a shorter deﬁnition for $B^{LW}_\mathbf{p}$ operator was given as a projector for a single plaquette in \cite{KK} (see \cite{Kon} for more explicit description).
In the following calculation, we will use symmetric $6j$-symbols that are the  factor proportionality when a bubble is removed at a vertex.
\[
%
\]

By $F$-move, the relation between these two kinds of $6j$-symbols is $G^{abc}_{kmn}=\frac{\theta(m,c,k)}{d_c}F^{amk}_{b;cn}=\frac{\sqrt{d_md_k}}{\sqrt{d_c}}F^{amk}_{b;cn}$. Moreover, we will use a graph to represent the symmetric $6j$-symbols $G^{abc}_{kmn}$.
\[
%
\]

Using symmetric $6j$-symbols, we can derive an explicit formula for $B^\mu_\mathbf{p}$. Note that each vertex contributes one symmetric $6j$-symbol and one has
\[
%
  \]
\end{lemma}

\begin{proof}
 By the pentagon equation \eqref{eqn:pentagon},
we have the pentagon equation for $G^{abc}_{kmn}$'s,
$$\sum\limits_nG^{bqp}_{dcn}G^{afq}_{dne}G^{aen}_{cbm}=\frac{\sqrt{d_cd_d}}{\sqrt{d_p}}G^{afq}_{pbm}G^{mfp}_{dce},$$
  i.e.,
  \[
  %
  \]

\noindent
 Note that there is an identity relating $G^{abc}_{kmn}$ to the tetrahedron.
  \[
  %
  \]
  By the tetrahedral symmetry, we obtain
  \[
  %
  \]
After plugging this into the pentagon equation and applying the rotation invariance of $G^{abc}_{kmn}$'s, the lemma is proven.
\end{proof}

\section{$C^*$-Quantum groupoids}

\subsection{Basic definition}

In this section we review basic properties of finite dimensional $C^*$-quantum
groupoids. The concept of quantum groupoid (weak Hopf algebra) is weakening the
constraint on the unit and counit as in Hopf algebras. Readers wanting more details
should consult Nikshych and Vainerman's survey article \cite{NV}.

Throughout this paper we use Sweedler's notation for comultiplication,
writing $\Delta(b)=\sum_{(b)}b_{(1)}\otimes b_{(2)}$. For simplicity, we shall suppress
the summation symbol and write $\Delta(b)=b_{(1)}\otimes b_{(2)}$ when no confusion occurs.
Using coassocitivity the iterated coproduct $\Delta^2(b)=(id\otimes\Delta)\Delta(b)=(\Delta\otimes id)\Delta(b)=b_{(1)}\otimes(b_{(2)})_{(1)}\otimes(b_{(2)})_{(2)}$
can be simply witten as $b_{(1)}\otimes b_{(2)}\otimes b_{(3)}$.

\begin{definition}
A finite quantum groupoid $H$ is a finite dimensional vector space with the structures
of an associative algebra $(H,m,\eta)$ with multiplication $m:H\otimes H\rightarrow  H$
and unit $\eta\in H$ and a coassociative coalgebra $(H,\Delta,\varepsilon)$ with comultiplication
$\Delta:H\rightarrow H\otimes H$ and counit $\mathbb{C}\rightarrow H$ such that:\\
(1) The comultiplication is an algebra homomorphism ($\Delta(ab)=\Delta(a)\Delta(b)$) such that
$$(\Delta\otimes id)\Delta(\eta)=(\Delta(\eta)\otimes\eta)(\eta\otimes\Delta(\eta))=(\eta\otimes\Delta(\eta))(\Delta(\eta)\otimes\eta)$$
(2) The counit is a linear map satisfying the identity:
$$\varepsilon(abc)=\varepsilon(ab_{(1)})\varepsilon(b_{(2)}c)=\varepsilon(ab_{(2)})\varepsilon(b_{(1)}c),~\forall a,b,c\in H$$
(3) There is a linear map $S:H\rightarrow H$, called an antipode, such that, for $\forall h\in H$,
$$m(id\otimes S)\Delta(h)=(\varepsilon\otimes id)(\Delta(\eta)(h\otimes\eta)),~i.e.~h_{(1)}S(h_{(2)})=\varepsilon(\eta_{(1)}h)\eta_{(2)},$$
$$m(S\otimes id)\Delta(h)=(id\otimes\varepsilon)((\eta\otimes h)\Delta(\eta)),~i.e.~S(h_{(1)})h_{(2)}=\varepsilon(h\eta_{(2)})\eta_{(1)},$$
$$m(m\otimes id)(S\otimes id\otimes S)(\Delta\otimes id)\Delta(h)=S(h),~i.e.~S(h_{(1)})h_{(2)}S(h_{(3)})=S(h).$$
A quantum groupoid $H$ is said to be a $C^*$-quantum groupoid if it is a $C^*$-algebra
and $\Delta$ is a $*$-homomorphism (\cite{NV}). Namely, its $*$-structure $*:H\rightarrow H$ satisfies $\Delta(h^*)=\Delta(h)^*$. 
\end{definition}

From the definition, one can see that a quantum groupoid is a Hopf algebra if and only if the comultiplication is unit-preserving, and if and only if the counit is an algebra homomorphism. 
  
The set of axioms of quantum groupoids is self-dual. This allows to define a natural
quantum groupoid structure on the dual space $H^*$ by:
$$\langle\phi\psi,h\rangle=\langle\phi\otimes\psi,\Delta(h)\rangle$$
$$\langle\widehat{\Delta}(\phi), h\otimes g\rangle=\langle\phi, hg\rangle$$
$$\langle\widehat{S}(\phi), h\rangle=\langle\phi, S(h)\rangle$$
for all $h,g\in H$ and $\phi, \psi\in H^*$. The unit $\widehat{\eta}$ of $H^*$ is $\varepsilon$
and the counit $\widehat{\varepsilon}$ is given by $\phi\mapsto\phi(\eta)$. One can show that the dual space $H^*$ of a $C^*$-quantum groupoid is also a $C^*$-quantum gorupoid with the $*$-structure given by $\langle \phi^*,h\rangle=\overline{\langle\phi,S(h)^*\rangle}$.

The linear maps defined in (3) are called target and source counital maps
and denoted by $\varepsilon_t$ and $\varepsilon_s$ respectively,
$$\varepsilon_t(h):=(\varepsilon\otimes id)(\Delta(\eta)(h\otimes\eta))=\varepsilon(\eta_{(1)}h)\eta_{(2)}$$
$$\varepsilon_s(h):=(id\otimes\varepsilon)((\eta\otimes h)\Delta(\eta))=\varepsilon(h\eta_{(2)})\eta_{(1)}.$$

\subsection{Representation theory of $C^*$-quantum groupoids}
For a quantum groupoid $H$, let $Rep(H)$ be the category of finite dimensional $H$-modules.
Similar to Hopf algebras, $Rep(H)$ has a structure of a tensor category with duality.
For objects $U,V$ of $Rep(H)$, their tensor product is defined to be
$$U\otimes V=\Delta(\eta)\cdot(U\otimes_\mathbb{C}V)$$
where $\otimes_\mathbb{C}$ means the usual tensor product between vector spaces.
The associtivity isomorphisms are the standard ones $(U\otimes V)\otimes W\cong U\otimes(V\otimes W)$.
The target counital subalgebra $H_t$ with an $H$-module structure $h\cdot z=\varepsilon_t(hz), \forall h\in H, z\in H_t$,
play the role of tensor unit object in $Rep(H)$. For any object $V$ of $Rep(H)$, its left dual is defined to be
$V^*=Hom_\mathbb{C}(V,\mathbb{C})$ with an $H$ action given by $(h\cdot\phi)(v)=\phi(S(h)\cdot v)$, for $h\in H, \phi\in V^*, v\in V$.
\begin{prop}
 \cite{NV} The category $Rep(H)$ is a tensor category with duality.
\end{prop}
Note that the unit object $H_t$ may not be simple, i.e. it may decompose as direct sum of
simple objects. If this happens, $Rep(H)$ is a multitensor category. Let $Z(H)$ be the center of $H$.
$H_t$ is simple if and only if $H_t\cap Z(H)=\mathbb{C}$. 

For a $C^*$-quantum groupoid $H$, a unitary representation of $H$ is understood to be a finite dimensional Hilbert space $V$
such that the scalar product $(~,~)_V$ satisfies $(v,h\cdot w)_V=(h^*\cdot v, w)_V$ for $v,w\in V$ and $h\in H$.
The tensor product between two unitary representation is defined as above.
The tensor unit is $H_t$ equipped with scalar product $(z,w)_{H_t}=\varepsilon(zw^*)$.
The dual $V^*$ of a unitary representation $V$ is the conjugate Hilbert space.
The action of $H$ on $V^*$ is $h\cdot\overline{v}=\overline{S(h)^*v}$ for $\overline{v}\in V^*$ and $h\in H$.
$V^*$ is equipped with a scalar product $(\overline{v},\overline{w})=(w,g\cdot v)$,
where is the canonical group-like element of $H$. Explicitly, $g$ satisfies $\Delta(g)=\Delta(\eta)(g\otimes g)$ and $S^2(h)=ghg^{-1}$ for $h\in H$.
\begin{prop}
 \cite{NV} For a $C^*$-quantum groupoid $H$, the category $URep(H)$ of unitary representations is a unitary (multi)fusion category.
\end{prop}

\subsection{Quantum double}
Similar to Hopf algebras, one can define the quantum double $D(H)$ for a quantum groupoid $H$ as follows.
On the vector space $H^{*cop}\otimes H$, a multiplication is given by
$$(\alpha\otimes x)(\beta\otimes y)=\alpha\beta_{(2)}\otimes x_{(2)}y\langle\beta_{(3)},S^{-1}(x_{(3)})\rangle\langle\beta_{(1)},x_{(1)}\rangle$$
Then one can verify the linear span $J$ of the elements
\begin{align}
\alpha(z\rightharpoonup\varepsilon)&\otimes x-\alpha\otimes zx,~z\in H_t,~x\in H,~\alpha\in H^* \label{eqn:target}\\
\alpha(\varepsilon\leftharpoonup w)&\otimes x-\alpha\otimes wx,~w\in H_s,~x\in H,~\alpha\in H^* \label{eqn:source}
\end{align}
is a two-sided ideal in $H^{*cop}\otimes H$. Here the action $\rightharpoonup$ is defined to be $z\rightharpoonup \varepsilon=\langle\varepsilon_{(2)}, z\rangle\varepsilon_{(1)}$ and the action $\leftharpoonup$ is defined by $\varepsilon\leftharpoonup w=\langle\varepsilon_{(1)}, w\rangle\varepsilon_{(2)}$.

$D(H)$ is defined to be the quotient $H^{*cop}\otimes H/J$
and $[\alpha\otimes x]$ denote the equivalence class of $\alpha\otimes x$.
The $D(H)$ is a quantum groupoid with unit $[\varepsilon\otimes\eta]$ and the comultiplication, counit, and antipode are given by
\begin{eqnarray*}
  \Delta([\alpha\otimes x])&=&[\alpha_{(1)}\otimes x_{(1)}]\otimes[\alpha_{(2)}\otimes x_{(2)}]\langle\alpha_{(3)},S^{-1}(x_{(3)})\rangle\langle\alpha_{(1)},x_{(1)}\rangle\\
  \varepsilon([\alpha\otimes x])&=&\varepsilon(x)\alpha(\eta)\\
  S([\alpha\otimes x])&=&[S^{-1}(\alpha_{(2)})\otimes S(x_{(2)})]\langle\alpha_{(3)},S^{-1}(x_{(3)})\rangle\langle\alpha_{(1)},x_{(1)}\rangle
\end{eqnarray*}
Both $H$ and $H^*$ can be embedded into $D(H)$ as sub quantum groupoids by
$h\mapsto[\varepsilon\otimes h]$ and $\alpha\mapsto[\alpha\otimes\eta]$.
Now consider the sets of cocommutative elements in $H$ and $H^*$.
They are $Cocom{H}=\{h\in H~|~\Delta(h)=\Delta^{op}(h)\}$ and $Cocom(H^*)=\{\alpha\in H^*~|~\alpha(xy)=\alpha(yx),~\forall x,y\in H\}$.
\begin{lemma}\label{prop:cocomutative}
  For $h\in Cocom{H}$ and $\alpha\in Cocom{H^*}$, $[\varepsilon\otimes h][\alpha\otimes\eta]=[\alpha\otimes h]=[\alpha\otimes\eta][\varepsilon\otimes h]$.
\end{lemma}
\begin{proof}
  \begin{eqnarray*}
    [\varepsilon\otimes h][\alpha\otimes\eta]&=&[\alpha_{(2)}\otimes h_{(2)}]\langle\alpha_{(3)},S^{-1}(h_{(3)})\rangle\langle\alpha_{(1)},h_{(1)}\rangle\\
    &=&[\alpha_{(1)}\otimes h_{(1)}]\langle\alpha_{(2)},S^{-1}(h_{(2)})\rangle\langle\alpha_{(3)},h_{(3)}\rangle\\
    &=&[\alpha_{(1)}\otimes h_{(1)}]\langle\alpha_{(2)},h_{(3)}S^{-1}(h_{(2)})\rangle\\
    &=&[\alpha\leftharpoonup h_{(3)}S^{-1}(h_{(2)})\otimes h_{(1)}]\\
    &=&[\alpha(\varepsilon\leftharpoonup h_{(3)}S^{-1}(h_{(2)}))\otimes h_{(1)}]\\
    &=&[\alpha\otimes h_{(3)}S^{-1}(h_{(2)})h_{(1)}]\\
    &=&[\alpha\otimes h]
  \end{eqnarray*}
 The other equality can be verified similarly.
\end{proof}

\section{Kitaev models based on $C^*$-quantum groupoids}

Let $H$ be a $C^*$-quantum gorupoid. Given an oriented lattice $\Gamma$ on an oriented compact surface $\Sigma$.
Then the space
$$\mathcal{L}^K=\bigotimes\limits_{edges}H$$
is the Hilbert space of the model. As in \cite{BCMA}, we consider the following operators:
for all $h,x\in H$ and $\alpha\in H^*$,
\begin{eqnarray*}
  L^h_{+}(x):&=&hx,\\
  L^h_{-}(x):&=&xS(h),\\
  T^\alpha_{+}(x):&=&\langle\alpha,x_{(2)}\rangle x_{(1)},\\
  T^\alpha_{+}(x):&=&\langle\alpha,S^{-1}(x_{(1)})\rangle x_{(2)}
\end{eqnarray*}

\begin{figure}[t]
    \begin{tabular}{cc}
    \begin{minipage}[t]{0.5\textwidth}
    \centering
    \begin{tikzpicture}
  \draw[line width=0.5mm] (0,-1)--(0,1);
  \draw[->,line width=0.7mm] (0,0)--(0,0.1);
  \node (L+) at (0,1.4) {$L^h_{+}$};
  \node (L-) at (0,-1.4) {$L^h_{-}$};
  \node (T+) at (0.8,0) {$T^\alpha_{+}$};
  \node (T-) at (-0.8,0) {$T^\alpha_{-}$};
\end{tikzpicture}
    \caption{\it Kitaev convention}
    \label{fig:Kitaev convention}
    \end{minipage}
    \begin{minipage}[t]{0.5\textwidth}
    \centering
    \begin{tikzpicture}
  \draw[line width=0.5mm] (0,-0.83) arc(-60:60:1);
  \draw[line width=0.5mm] (0,-0.83) arc(240:120:1);
  \draw[->,line width=0.7mm] (0.5,0)--(0.5,0.1);
  \draw[->,line width=0.7mm] (-0.5,0)--(-0.5,-0.1);
  \node (p) at (0,0) {$\mathbf{p}$};
  \node (x1) at (1,0) {$x_1$};
  \node (x2) at (-1,0) {$x_2$};
  \node (v) at (0,1.1) {$\mathbf{v}$};
\end{tikzpicture}
    \caption{\it Site $(\mathbf{v},\mathbf{p})$}
    \label{fig:bigon}
    \end{minipage}
\end{tabular}
\end{figure}

Following Kitaev's convention shown in Figure~\ref{fig:Kitaev convention}, a pair of local operators $A_h$ and $B_\alpha$
are defined at a site $(\mathbf{v},\mathbf{p})$ of the lattice $\Gamma$.
In particular, for the site shown in Figure~\ref{fig:bigon}, the local operators are
$$A_h(\mathbf{v},\mathbf{p})=\sum\limits_{(h)}L^{h_{(1)}}_{+}\otimes L^{h_{(2)}}_{-},$$
$$B_\alpha(\mathbf{v},\mathbf{p})=\sum\limits_{(\alpha)}T^{\alpha_{(1)}}_{-}\otimes T^{\alpha_{(2)}}_{-}.$$

\begin{prop}
  The operators $A_h(\mathbf{v},\mathbf{p})$ and $B_\alpha(\mathbf{v},\mathbf{p})$ satisfy the commutation relation:
  \begin{eqnarray*}
    A_hB_\alpha&=&\sum\limits_{(h),(\alpha)}B_{\alpha_{(2)}}A_{h_{(2)}}\langle\alpha_{(3)},S^{-1}(h_{(3)})\rangle\langle\alpha_{(1)},h_{(1)}\rangle\\
    B_{\alpha(z\rightharpoonup\varepsilon)}A_h&=&B_\alpha A_{zh}\\
    B_{\alpha(\varepsilon\leftharpoonup w)}A_h&=&B_\alpha A_{wh}
  \end{eqnarray*}
  where $z\in H_t, w\in H_s$. Hence, we have an algebra homomorphism
\begin{eqnarray*}
D(H) &\rightarrow& End(H\otimes H)\\
{[\alpha\otimes h]} &\mapsto&B_\alpha A_h
\end{eqnarray*}
\end{prop}
\begin{proof}
  The multiplication relation follows by computation as \cite{BCMA}. Here we verify the relations \eqref{eqn:target} and \eqref{eqn:source}.
  For $x_1,x_2\in H$,
  $$B_{\alpha(z\rightharpoonup\varepsilon)}A_h(x_1\otimes x_2)=B_{\alpha(z\rightharpoonup\varepsilon)}(h_{(1)}x_1\otimes x_2S(h_{(2)}))=B_{\alpha(z\rightharpoonup\varepsilon)}(a\otimes b)$$
  where we used the temporary abbreviation $a=h_{(1)}x_1$ and $b=x_2 S(h_{(2)})$. Then
  \begin{eqnarray*}
  B_{\alpha(z\rightharpoonup\varepsilon)}A_h(x_1\otimes x_2)&=&\langle\alpha_{(1)}(z\rightharpoonup\varepsilon)_{(1)},S^{-1}(a_{(1)})\rangle a_{(2)}\otimes\langle\alpha_{(2)}(z\rightharpoonup\varepsilon)_{(2)},S^{-1}(b_{(1)})\rangle b_{(2)}\\
  &=&\langle (S^{-1}(z)\rightharpoonup\varepsilon)\alpha_{(1)},S^{-1}(a_{(1)})\rangle a_{(2)}\otimes\langle\alpha_{(2)},S^{-1}(b_{(1)})\rangle b_{(2)}\\
  &=&\langle S^{-1}(z)\rightharpoonup\varepsilon,S^{-1}(a_{(1)})\rangle\langle\alpha_{(2)},S^{-1}(a_{(1)})\rangle a_{(3)}\otimes\langle\alpha_{(2)},S^{-1}(b_{(1)})\rangle b_{(2)}\\
  &=&\langle \varepsilon,S^{-1}(a_{(1)})S^{-1}(z)\rangle\langle\alpha_{(2)},S^{-1}(a_{(1)})\rangle a_{(3)}\otimes\langle\alpha_{(2)},S^{-1}(b_{(1)})\rangle b_{(2)}\\
  &=&\langle \alpha_{(2)},S^{-1}(za_{(1)})\rangle a_{(2)}\otimes\langle\alpha_{(2)},S^{-1}(b_{(1)})\rangle b_{(2)}\\
  &=&\langle \alpha_{(2)},S^{-1}(z_{(1)}a_{(1)})\rangle z_{(2)}a_{(2)}\otimes\langle\alpha_{(2)},S^{-1}(b_{(1)})\rangle b_{(2)}\\
  &=&B_\alpha(za\otimes b)\\
  &=&B_\alpha(z_{(1)}h_{(1)}x_1\otimes x_2S(z_{(2)}h_{(2)}))\\
  &=&B_\alpha A_{zh}(x_1\otimes x_2)
  \end{eqnarray*}
Similarly, one can check that $B_{\alpha(\varepsilon\leftharpoonup w)}A_h(x_1\otimes x_2)=B_\alpha A_{wh}(x_1\otimes x_2)$.
\end{proof}

Under this homomorphism, we have the following corollary from Lemma \ref{prop:cocomutative}.
\begin{cor}
  For $h\in Cocom(H)$ and $\alpha\in Cocom(H^*)$, $B_\alpha A_h=A_h B_\alpha$.
\end{cor}
For $h\in Cocom(H)$, the vertex operator $A_h(\mathbf{v}):=A_h(\mathbf{v},\mathbf{p})$ does not depend on the choice of the adjacent faces $\mathbf{p}$'s. So does the plaquette operator $B_\alpha(\mathbf{p}):=B_\alpha(\mathbf{v},\mathbf{p})$ for $\alpha\in Cocom(H^*)$. Indeed, by cocommutativity, the coproduct can be permuted cyclically.
\begin{prop}\label{prop:ABoperators}
  Given $h\in Cocom(H)$ and $\alpha\in Cocom(H^*)$ in a quantum groupoid $H$, all operators $A_h(\mathbf{v})$, $B_\alpha(\mathbf{p})$ commute with each other.
\end{prop}

\section{Kitaev and Kong's quantum groupoids}

In the following sections, we assume the fusion categories are self-dual and multiplicity free. In addition, we assume their Frobenius-Schur indicators are trivial which allows the $F$-matrices endowed with certain symmetry.

\subsection{$C^*$-quantum groupoid $H_\mathcal{C}$ for unitary fusion category $\mathcal{C}$.}
In \cite{KK} Kitaev and Kong constructed a $C^*$-quantum groupoid $H_\mathcal{C}$ from a unitary fusion category $\mathcal{C}$.
Here we give an equivalent definition. As a vector space,
$$H_\mathcal{C}=\bigoplus\limits_{a,b,c,d,i\in Irr(\mathcal{C})} Hom(b,a\otimes i)\otimes Hom(i\otimes d,c)$$
and it is spanned by the basis $e^{ab}_{i;cd}$, where $i,a,b,c,d$ are
representatives of simple objects of $\mathcal{C}$. The basis vector $e^{ab}_{i;cd}$ has a graphical interpretation.
\[
\begin{tikzpicture}[scale=0.5]
  \node (e) at (0,0) {$e^{ab}_{i;cd}=$};
  \begin{scope}[xshift=3cm]
  \draw[line width=0.5mm] (-1,0)--(1,0);
  \draw[line width=0.5mm] (-1,1)--(-1,-1);
  \draw[line width=0.5mm] (1,1)--(1,-1);
  \node (a) at (-1.3,1.3) {$a$};
  \node (b) at (-1.3,-1.3) {$b$};
  \node (c) at (1.3,1.3) {$c$};
  \node (d) at (1.3,-1.3) {$d$};
  \node (i) at (0,0.5) {$i$};
  \end{scope}
\end{tikzpicture}
\]
These basis vectors are linearly transformed from the basis used in \cite{KK} by $F$-moves. The quantum groupoid structure is given as follows in terms of the basis $\{e^{ab}_{i;cd}\}$.

\noindent
Multiplication
\[
\begin{tikzpicture}[scale=0.5]
  \begin{scope}[xshift=-7.5cm]
  \draw[line width=0.5mm] (-1,0)--(1,0);
  \draw[line width=0.5mm] (-1,1)--(-1,-1);
  \draw[line width=0.5mm] (1,1)--(1,-1);
  \node (a) at (-1.3,1.3) {$a$};
  \node (b) at (-1.3,-1.3) {$b$};
  \node (c) at (1.3,1.3) {$c$};
  \node (d) at (1.3,-1.3) {$d$};
  \node (i) at (0,0.5) {$i$};
  \end{scope}
  \node (dot) at (-5.5,0) {$\cdot$};
  \begin{scope}[xshift=-3.5cm]
  \draw[line width=0.5mm] (-1,0)--(1,0);
  \draw[line width=0.5mm] (-1,1)--(-1,-1);
  \draw[line width=0.5mm] (1,1)--(1,-1);
  \node (a) at (-1.3,1.3) {$a'$};
  \node (b) at (-1.3,-1.3) {$b'$};
  \node (c) at (1.3,1.3) {$c'$};
  \node (d) at (1.3,-1.3) {$d'$};
  \node (i) at (0,0.5) {$i'$};
  \end{scope}
  \node (=) at (0,0) {$=\frac{\delta_{c,a'}\delta_{d,b'}\delta_{i,i'}}{\sqrt{d_i}}$};
  \begin{scope}[xshift=4cm]
  \draw[line width=0.5mm] (-1,0)--(1,0);
  \draw[line width=0.5mm] (-1,1)--(-1,-1);
  \draw[line width=0.5mm] (1,1)--(1,-1);
  \node (a) at (-1.3,1.3) {$a$};
  \node (b) at (-1.3,-1.3) {$b$};
  \node (c) at (1.3,1.3) {$c'$};
  \node (d) at (1.3,-1.3) {$d'$};
  \node (i) at (0,0.5) {$i'$};
  \end{scope}
\end{tikzpicture}
\]
\noindent
Unit
\[
\begin{tikzpicture}[scale=0.5]
  \node (unit) at (0,-0.3) {$\eta=\sum\limits_{a,b,i}\sqrt{d_i}$};
  \begin{scope}[xshift=4cm]
  \draw[line width=0.5mm] (-1,0)--(1,0);
  \draw[line width=0.5mm] (-1,1)--(-1,-1);
  \draw[line width=0.5mm] (1,1)--(1,-1);
  \node (a) at (-1.3,1.3) {$a$};
  \node (b) at (-1.3,-1.3) {$b$};
  \node (c) at (1.3,1.3) {$a$};
  \node (d) at (1.3,-1.3) {$b$};
  \node (i) at (0,0.5) {$i$};
  \end{scope}
\end{tikzpicture}
\]
\noindent
Comultiplication
\[
\begin{tikzpicture}[scale=0.5]
  \begin{scope}[xshift=-3.5cm]      
  \node (delta) at (-3,0) {$\Delta$};
  \node (<) at (-2.2,0) {$\Bigg($};
  \node (>) at (2.2,0) {$\Bigg)$};
  \draw[line width=0.5mm] (-1,0)--(1,0);
  \draw[line width=0.5mm] (-1,1)--(-1,-1);
  \draw[line width=0.5mm] (1,1)--(1,-1);
  \node (a) at (-1.3,1.3) {$a$};
  \node (b) at (-1.3,-1.3) {$b$};
  \node (c) at (1.3,1.3) {$c$};
  \node (d) at (1.3,-1.3) {$d$};
  \node (i) at (0,0.5) {$i$};
  \end{scope}
  \node (=) at (0,-0.5) {$=\sum\limits_{\substack{j,k\\p,q}}$};
  \begin{scope}[xshift=5.8cm]             
    \node (<) at (-4.3,0) {$\Bigg\langle$};
    \node (>) at (4.3,0) {$\Bigg\rangle$};
    \begin{scope}[xshift=-2cm]
  \draw[line width=0.5mm] (-1,0)--(1,0);
  \draw[line width=0.5mm] (-1,1)--(-1,-1);
  \draw[line width=0.5mm] (1,1)--(1,-1);
  \node (a) at (-1.3,1.3) {$a$};
  \node (b) at (-1.3,-1.3) {$b$};
  \node (c) at (1.3,1.3) {$c$};
  \node (d) at (1.3,-1.3) {$d$};
  \node (i) at (0,0.5) {$i$};
  \node (hat) at (0,1.9) {$\wedge$};
    \end{scope}
    \node (comma) at (0,-0.8) {$,$};
    \begin{scope}[xshift=2cm]
  \draw[line width=0.5mm] (-1,0.5)--(1,0.5);
  \draw[line width=0.5mm] (-1,-0.5)--(1,-0.5);
  \draw[line width=0.5mm] (-1,1.5)--(-1,-1.5);
  \draw[line width=0.5mm] (1,1.5)--(1,-1.5);
  \node (a) at (-1.3,1.8) {$a$};
  \node (b) at (-1.3,-1.8) {$b$};
  \node (c) at (1.3,1.8) {$c$};
  \node (d) at (1.3,-1.8) {$d$};
  \node (p) at (-1.3,0) {$p$};
  \node (q) at (1.3,0) {$q$};
  \node (j) at (0,1) {$j$};
  \node (k) at (0,-1) {$k$};
    \end{scope}
  \end{scope}
  \begin{scope}[xshift=14.6cm]  
    \begin{scope}[xshift=-2.2cm]
  \draw[line width=0.5mm] (-1,0)--(1,0);
  \draw[line width=0.5mm] (-1,1)--(-1,-1);
  \draw[line width=0.5mm] (1,1)--(1,-1);
  \node (a) at (-1.3,1.3) {$a$};
  \node (b) at (-1.3,-1.3) {$p$};
  \node (c) at (1.3,1.3) {$c$};
  \node (d) at (1.3,-1.3) {$q$};
  \node (i) at (0,0.5) {$j$};
    \end{scope}
    \node (tensor) at (0,0) {$\otimes$};
    \begin{scope}[xshift=2.2cm]
  \draw[line width=0.5mm] (-1,0)--(1,0);
  \draw[line width=0.5mm] (-1,1)--(-1,-1);
  \draw[line width=0.5mm] (1,1)--(1,-1);
  \node (a) at (-1.3,1.3) {$p$};
  \node (b) at (-1.3,-1.3) {$b$};
  \node (c) at (1.3,1.3) {$q$};
  \node (d) at (1.3,-1.3) {$d$};
  \node (i) at (0,0.5) {$k$};
    \end{scope}
  \end{scope}
\end{tikzpicture}
\]
\noindent
Here, the hat means the dual basis in the dual space $H^*_\mathcal{C}$ and the pairing means the coefficient after replacing the ladder by the linear combination in terms of $e^{ab}_{i;cd}$'s by $F$-moves.\\
\noindent
Counit
\[
\begin{tikzpicture}[scale=0.5]
  \begin{scope}[xshift=-4.8cm]
  \node (epsilon) at (-2.8,0) {$\varepsilon$};
  \node (<) at (-2.2,0) {$\Bigg($};
  \node (>) at (2.2,0) {$\Bigg)$};
  \draw[line width=0.5mm] (-1,0)--(1,0);
  \draw[line width=0.5mm] (-1,1)--(-1,-1);
  \draw[line width=0.5mm] (1,1)--(1,-1);
  \node (a) at (-1.3,1.3) {$a$};
  \node (b) at (-1.3,-1.3) {$b$};
  \node (c) at (1.3,1.3) {$c$};
  \node (d) at (1.3,-1.3) {$d$};
  \node (i) at (0,0.5) {$i$};
  \end{scope}
  \node (=) at (0,0) {$=\delta_{a,b}\delta_{c,d}\delta_{i,\textbf{1}}$};
\end{tikzpicture}
\]
\noindent
Antipode
\[
\begin{tikzpicture}[scale=0.5]
  \begin{scope}[xshift=-4.3cm]
  \node (S) at (-2.9,0) {$S$};
  \node (<) at (-2.2,0) {$\Bigg($};
  \node (>) at (2.2,0) {$\Bigg)$};
  \draw[line width=0.5mm] (-1,0)--(1,0);
  \draw[line width=0.5mm] (-1,1)--(-1,-1);
  \draw[line width=0.5mm] (1,1)--(1,-1);
  \node (a) at (-1.3,1.3) {$a$};
  \node (b) at (-1.3,-1.3) {$b$};
  \node (c) at (1.3,1.3) {$c$};
  \node (d) at (1.3,-1.3) {$d$};
  \node (i) at (0,0.5) {$i$};
  \end{scope}
  \node (=) at (0,0) {$=\frac{\sqrt{d_bd_c}}{\sqrt{d_ad_d}}$};
  \begin{scope}[xshift=3.3cm]
  \draw[line width=0.5mm] (-1,0)--(1,0);
  \draw[line width=0.5mm] (-1,1)--(-1,-1);
  \draw[line width=0.5mm] (1,1)--(1,-1);
  \node (a) at (-1.3,1.3) {$d$};
  \node (b) at (-1.3,-1.3) {$c$};
  \node (c) at (1.3,1.3) {$b$};
  \node (d) at (1.3,-1.3) {$a$};
  \node (i) at (0,0.5) {$i$};
  \end{scope}
\end{tikzpicture}
\]
\noindent
Star
\[
\begin{tikzpicture}[scale=0.5]
  \begin{scope}[xshift=-3.3cm]
  \node (<) at (-2.2,0) {$\Bigg($};
  \node (>) at (2.2,0) {$\Bigg)^*$};
  \draw[line width=0.5mm] (-1,0)--(1,0);
  \draw[line width=0.5mm] (-1,1)--(-1,-1);
  \draw[line width=0.5mm] (1,1)--(1,-1);
  \node (a) at (-1.3,1.3) {$a$};
  \node (b) at (-1.3,-1.3) {$b$};
  \node (c) at (1.3,1.3) {$c$};
  \node (d) at (1.3,-1.3) {$d$};
  \node (i) at (0,0.5) {$i$};
  \end{scope}
  \node (=) at (0,0) {$=$};
  \begin{scope}[xshift=2.5cm]
  \draw[line width=0.5mm] (-1,0)--(1,0);
  \draw[line width=0.5mm] (-1,1)--(-1,-1);
  \draw[line width=0.5mm] (1,1)--(1,-1);
  \node (a) at (-1.3,1.3) {$c$};
  \node (b) at (-1.3,-1.4) {$d$};
  \node (c) at (1.3,1.3) {$a$};
  \node (d) at (1.3,-1.4) {$b$};
  \node (i) at (0,0.5) {$i$};
  \end{scope}
\end{tikzpicture}
\]

Everything in checking the satisfactory for the axioms of quantum groupoids is straightforward except that the comuplication is an algebra homomorphism. Indeed, it needs the unitarity of $F$-matrices. Thus a unitary fusion category $\mathcal{C}$ is essential for the quantum groupoid structure of $H_\mathcal{C}$. The interested readers can consult \cite{PZ} (Appendix A) for details.

\begin{lemma}\label{prop:comultiplication}
  The comultiplication can be written in terms of symmetric $6j$-symbols
\[
\begin{tikzpicture}[scale=0.5]
  \begin{scope}[xshift=-5.2cm]      
  \node (delta) at (-3,0) {$\Delta$};
  \node (<) at (-2.2,0) {$\Bigg($};
  \node (>) at (2.2,0) {$\Bigg)$};
  \draw[line width=0.5mm] (-1,0)--(1,0);
  \draw[line width=0.5mm] (-1,1)--(-1,-1);
  \draw[line width=0.5mm] (1,1)--(1,-1);
  \node (a) at (-1.3,1.3) {$a$};
  \node (b) at (-1.3,-1.3) {$b$};
  \node (c) at (1.3,1.3) {$c$};
  \node (d) at (1.3,-1.3) {$d$};
  \node (i) at (0,0.5) {$i$};
  \end{scope}
  \node (=) at (0,-0.4) {$=\sum\limits_{\substack{j,k\\p,q}}\frac{\sqrt{d_id_jd_k}}{\sqrt{d_ad_bd_cd_d}}$};
  \begin{scope}[xshift=5.5cm]
  \node (<) at (-2.6,0.2) {$\big\langle$};
  \node (>) at (2.6,0.2) {$\big\rangle$};
  \draw[line width=0.5mm] (0,1)--(0,2);
  \draw[line width=0.5mm] (0.86,-0.5)--(1.73,-1);
  \draw[line width=0.5mm] (-0.86,-0.5)--(-1.73,-1);
  \draw[line width=0.5mm] (1,0) arc(0:360:1);
  \node (a) at (-0.1,2.4) {$i$};
  \node (b) at (-2,-1.2) {$k$};
  \node (c) at (2,-1.2) {$j$};
  \node (m) at (1.6,0.2) {$a$};
  \node (n) at (-1.5,0.2) {$b$};
  \node (k) at (0,-1.5) {$p$};
  \end{scope}
  \begin{scope}[xshift=11.2cm]
  \node (<) at (-2.6,0.2) {$\big\langle$};
  \node (>) at (2.6,0.2) {$\big\rangle$};
  \draw[line width=0.5mm] (0,1)--(0,2);
  \draw[line width=0.5mm] (0.86,-0.5)--(1.73,-1);
  \draw[line width=0.5mm] (-0.86,-0.5)--(-1.73,-1);
  \draw[line width=0.5mm] (1,0) arc(0:360:1);
  \node (a) at (-0.1,2.4) {$i$};
  \node (b) at (-2,-1.2) {$j$};
  \node (c) at (2,-1.2) {$k$};
  \node (m) at (1.6,0.2) {$d$};
  \node (n) at (-1.5,0.2) {$c$};
  \node (k) at (0,-1.5) {$q$};
  \end{scope}
  \begin{scope}[xshift=18.1cm]  
    \begin{scope}[xshift=-2.2cm]
  \draw[line width=0.5mm] (-1,0)--(1,0);
  \draw[line width=0.5mm] (-1,1)--(-1,-1);
  \draw[line width=0.5mm] (1,1)--(1,-1);
  \node (a) at (-1.3,1.3) {$a$};
  \node (b) at (-1.3,-1.3) {$p$};
  \node (c) at (1.3,1.3) {$c$};
  \node (d) at (1.3,-1.3) {$q$};
  \node (i) at (0,0.5) {$j$};
    \end{scope}
    \node (tensor) at (0,0) {$\otimes$};
    \begin{scope}[xshift=2.2cm]
  \draw[line width=0.5mm] (-1,0)--(1,0);
  \draw[line width=0.5mm] (-1,1)--(-1,-1);
  \draw[line width=0.5mm] (1,1)--(1,-1);
  \node (a) at (-1.3,1.3) {$p$};
  \node (b) at (-1.3,-1.3) {$b$};
  \node (c) at (1.3,1.3) {$q$};
  \node (d) at (1.3,-1.3) {$d$};
  \node (i) at (0,0.5) {$k$};
    \end{scope}
  \end{scope}
\end{tikzpicture}
\]
\noindent
where $j,k$ must be admissible with $i$.
\end{lemma}
\begin{proof}
  Note that
  \[
  \begin{tikzpicture}[scale=0.5]
    \begin{scope}[xshift=-5.2cm]             
    \node (<) at (-4.3,0) {$\Bigg\langle$};
    \node (>) at (4.3,0) {$\Bigg\rangle$};
    \begin{scope}[xshift=-2cm]
  \draw[line width=0.5mm] (-1,0)--(1,0);
  \draw[line width=0.5mm] (-1,1)--(-1,-1);
  \draw[line width=0.5mm] (1,1)--(1,-1);
  \node (a) at (-1.3,1.3) {$a$};
  \node (b) at (-1.3,-1.3) {$b$};
  \node (c) at (1.3,1.3) {$c$};
  \node (d) at (1.3,-1.3) {$d$};
  \node (i) at (0,0.5) {$i$};
  \node (hat) at (0,1.9) {$\wedge$};
    \end{scope}
    \node (comma) at (0,-0.8) {$,$};
    \begin{scope}[xshift=2cm]
  \draw[line width=0.5mm] (-1,0.5)--(1,0.5);
  \draw[line width=0.5mm] (-1,-0.5)--(1,-0.5);
  \draw[line width=0.5mm] (-1,1.5)--(-1,-1.5);
  \draw[line width=0.5mm] (1,1.5)--(1,-1.5);
  \node (a) at (-1.3,1.8) {$a$};
  \node (b) at (-1.3,-1.8) {$b$};
  \node (c) at (1.3,1.8) {$c$};
  \node (d) at (1.3,-1.8) {$d$};
  \node (p) at (-1.3,0) {$p$};
  \node (q) at (1.3,0) {$q$};
  \node (j) at (0,1) {$j$};
  \node (k) at (0,-1) {$k$};
    \end{scope}
  \end{scope}
  \node (=) at (0,0) {$=$};
  \begin{scope}[xshift=3.2cm,scale=1.2]
  \draw[line width=0.5mm] (0,1)--(0,2);
  \draw[line width=0.5mm] (0.86,-0.5)--(1.73,-1);
  \draw[line width=0.5mm] (-0.86,-0.5)--(-1.73,-1);
  \draw[line width=0.5mm] (1,0) arc(0:360:1);
  \draw[line width=0.5mm] (2,0) arc(0:360:2);
  \node (a) at (1.8,1.6) {$b$};
  \node (b) at (-1.8,1.6) {$a$};
  \node (c) at (0,-2.5) {$p$};
  \node (m) at (-0.4,1.4) {$i$};
  \node (n) at (-1.5,-0.5) {$j$};
  \node (k) at (1.5,-0.5) {$k$};
  \node (p) at (1.2,0.7) {$d$};
  \node (q) at (-1.2,0.7) {$c$};
  \node (r) at (0,-0.6) {$q$};
  \end{scope}
  \node (quotient) at (6.6,0) {$\Bigg/$};
  \begin{scope}[xshift=9.9cm,scale=1.2]
  \draw[line width=0.5mm] (1,0)--(2,0);
  \draw[line width=0.5mm] (-1,0)--(-2,0);
  \draw[line width=0.5mm] (1,0) arc(0:360:1);
  \draw[line width=0.5mm] (2,0) arc(0:360:2);
  \node (a) at (0,2.4) {$a$};
  \node (c) at (0,-2.5) {$b$};
  \node (m) at (1.5,0.4) {$i$};
  \node (n) at (-1.5,0.4) {$i$};
  \node (k) at (0,-1.4) {$d$};
  \node (p) at (0,1.3) {$c$};
  \end{scope}
  \end{tikzpicture}
  \]
\noindent
The lemma follows by calculating these circle values by symmetric $6j$-symbols and $\theta(i,j,k)$. After factoring out one symmetric $6j$-symbols, one needs to evaluate the following circle
\[
\begin{tikzpicture}[scale=0.5]
  \begin{scope}[xshift=-6.7cm]
  \draw[line width=0.5mm] (2,0) arc(0:360:2);
  \draw[line width=0.5mm] (0,0)--(0,2);
  \draw[line width=0.5mm] (0,0)--(1.73,-1);
  \draw[line width=0.5mm] (0,0)--(-1.73,-1);
  \node (a) at (-0.4,1) {$i$};
  \node (b) at (-1,0) {$j$};
  \node (c) at (1.1,0) {$k$};
  \node (m) at (2.4,0.4) {$b$};
  \node (n) at (-2.4,0.4) {$a$};
  \node (k) at (0,-2.6) {$p$};
  \end{scope}
  \node (=) at (0,0) {$=F^{ajk}_{b,ip}\theta(i,a,b)\frac{\sqrt{d_jd_k}}{\sqrt{d_i}}$.};
\end{tikzpicture}
\]
\noindent
Note that $F^{ajk}_{b,ip}=\overline{F^{bkj}_{a,ip}}$. Combine this and the relation among $F$-symbols, $G$-symbols and the tetrahedron, one obtains
\[
\begin{tikzpicture}[scale=0.5]
  \begin{scope}[xshift=-4.7cm]
  \draw[line width=0.5mm] (2,0) arc(0:360:2);
  \draw[line width=0.5mm] (0,0)--(0,2);
  \draw[line width=0.5mm] (0,0)--(1.73,-1);
  \draw[line width=0.5mm] (0,0)--(-1.73,-1);
  \node (a) at (-0.4,1) {$i$};
  \node (b) at (-1,0) {$j$};
  \node (c) at (1.1,0) {$k$};
  \node (m) at (2.4,0.4) {$b$};
  \node (n) at (-2.4,0.4) {$a$};
  \node (k) at (0,-2.6) {$p$};
  \end{scope}
  \node (=) at (0,0) {$=\theta(i,j,k)$};
    \begin{scope}[xshift=4.7cm, yshift=-0.5]
  \node (<) at (-2.6,0.2) {$\big\langle$};
  \node (>) at (2.6,0.2) {$\big\rangle$};
  \draw[line width=0.5mm] (0,1)--(0,2);
  \draw[line width=0.5mm] (0.86,-0.5)--(1.73,-1);
  \draw[line width=0.5mm] (-0.86,-0.5)--(-1.73,-1);
  \draw[line width=0.5mm] (1,0) arc(0:360:1);
  \node (a) at (-0.1,2.4) {$i$};
  \node (b) at (-2,-1.2) {$k$};
  \node (c) at (2,-1.2) {$j$};
  \node (m) at (1.6,0.2) {$a$};
  \node (n) at (-1.5,0.2) {$b$};
  \node (k) at (0,-1.5) {$p$};
  \end{scope}
\end{tikzpicture}
\]
\end{proof}

In addition, $H_\mathcal{C}$ is equipped with a positive inner product $(x,y)=\langle\chi,x^*y\rangle$ for $x,y\in H_\mathcal{C}$.
Here $\chi\in H_\mathcal{C}^*$ is
\[
\begin{tikzpicture}[scale=0.5]
  \node (chi) at (0,-0.2) {$\chi=\sum\limits_{a,b,i}\frac{1}{\sqrt{d_i}}\frac{d_b}{d_a}$};
  \begin{scope}[xshift=4.3cm]
  \draw[line width=0.5mm] (-1,0)--(1,0);
  \draw[line width=0.5mm] (-1,1)--(-1,-1);
  \draw[line width=0.5mm] (1,1)--(1,-1);
  \node (a) at (-1.3,1.3) {$a$};
  \node (b) at (-1.3,-1.3) {$b$};
  \node (c) at (1.3,1.3) {$a$};
  \node (d) at (1.3,-1.3) {$b$};
  \node (i) at (0,0.5) {$i$};
  \node (hat) at (0,1.9) {$\wedge$};
  \end{scope}
\end{tikzpicture}
\]

With this inner product, $H_\mathcal{C}$ become a Hilbert space.
Indeed, it is easy to see that $\langle\chi,e^{ab}_{i;cd}\rangle=\langle\chi,(e^{ab}_{i;cd})^*\rangle$ is positive for any basis vector $e^{ab}_{i;cd}$ of $H_\mathcal{C}$.

\begin{example}
  An interesting and important example is the Fibonacci theory.
  In this theory, we have a unitary fusion category $\mathcal{F}$ with two simple objects $\textbf{1}$ and $\tau$.
  They are self dual and of quantum dimension $d_{\textbf{1}}=1$ and $d_\tau=\phi$. Here $\phi=\frac{1+\sqrt{5}}{2}$ is the golden ratio.
  The only nontrivial fusion rule is $\tau^2=\textbf{1}+\tau$. The nontrivial $6j$-symbols are
  $F^{\tau\tau\tau}_{\tau;\textbf{1}\textbf{1}}=\phi^{-1}$, $F^{\tau\tau\tau}_{\tau;\tau\textbf{1}}=\phi^{-\frac{1}{2}}$,
  $F^{\tau\tau\tau}_{\tau;\textbf{1}\tau}=\phi^{-\frac{1}{2}}$, $F^{\tau\tau\tau}_{\tau;\tau\tau}=-\phi^{-1}$.
  By the above definition, we have a 13 dimensional quantum groupoid $H_\mathcal{F}$ with basis
  $\{e^{\textbf{1}\textbf{1}}_{\textbf{1};\textbf{1}\textbf{1}},e^{\textbf{1}\textbf{1}}_{\textbf{1};\tau\tau},e^{\tau\tau}_{\textbf{1};\textbf{1}\textbf{1}},e^{\tau\tau}_{\textbf{1};\tau\tau},e^{\tau\tau}_{\tau;\tau\tau},e^{\tau\tau}_{\tau;\textbf{1}\tau},e^{\tau\tau}_{\tau;\tau\textbf{1}},e^{\textbf{1}\tau}_{\tau;\tau\tau},e^{\tau\textbf{1}}_{\tau;\tau\tau},e^{\tau\textbf{1}}_{\tau;\tau\textbf{1}},e^{\textbf{1}\tau}_{\tau;\textbf{1}\tau},e^{\tau\textbf{1}}_{\tau;\textbf{1}\tau},e^{\textbf{1}\tau}_{\tau;\tau\textbf{1}}\}$.
  By unitarity, it is a direct sum of two matrix algebras $H_\mathcal{F}=M_{2\times2}\oplus M_{3\times3}$.
  This algebra is actually $TLJ_4(ie^{\frac{2\pi i}{20}})$, the Temperley-Lieb-Jones algebra at the 5th root of unity.
  In fact, as an algebra, $H_\mathcal{F}$ is generated by
  \begin{eqnarray*}
    e_0&=&\phi e^{\textbf{1}\textbf{1}}_{\textbf{1};\textbf{1}\textbf{1}}+\phi^{\frac{3}{2}}e^{\textbf{1}\tau}_{\tau;\textbf{1}\tau},\\
    e_1&=&\phi^{-1}e^{\textbf{1}\textbf{1}}_{\textbf{1};\textbf{1}\textbf{1}}+\phi^{-\frac{1}{2}}e^{\textbf{1}\textbf{1}}_{\textbf{1};\tau\tau}+\phi^{-\frac{1}{2}}e^{\tau\tau}_{\textbf{1};\textbf{1}\textbf{1}}+e^{\tau\tau}_{\textbf{1};\tau\tau}+\phi^{-\frac{1}{2}}e^{\textbf{1}\tau}_{\tau;\textbf{1}\tau}+e^{\textbf{1}\tau}_{\tau;\tau\textbf{1}}+e^{\tau\textbf{1}}_{\tau;\textbf{1}\tau}+\phi^{\frac{1}{2}}e^{\tau\textbf{1}}_{\tau;\tau\textbf{1}},\\
    e_2&=&\phi e^{\textbf{1}\textbf{1}}_{\textbf{1};\textbf{1}\textbf{1}}+\phi^{\frac{1}{2}}e^{\tau\textbf{1}}_{\tau;\tau\textbf{1}}+e^{\tau\textbf{1}}_{\tau;\tau\tau}+e^{\tau\tau}_{\tau;\tau\textbf{1}}+\phi^{-\frac{1}{2}}e^{\tau\tau}_{\tau;\tau\tau}.
  \end{eqnarray*}
  They satisfy the relations $e_k^2=\phi e_k$ for $k=1,2,3$ and $e_0e_1e_0=e_0$, $e_1e_0e_1=e_1$, $e_1e_2e_1=e_1$, $e_2e_1e_2=e_2$, $e_0e_2=e_2e_0$ and $\phi^2e_1+\phi(e_0+e_2)-\phi(e_0e_1+e_1e_0+e_1e_2+e_2e_1)-\phi^2e_0e_2+e_0e_1e_2+e_2e_1e_0+\phi(e_1e_0e_2+e_2e_0e_1)-e_1e_0e_2e_1=1$.
\end{example}

If we start with the representation category $Rep(Q)$ of a general $C^*$-quantum groupoid $Q$, we can not expect that $H_{Rep(Q)}\cong Q$.
The following example provides a counter example.
\begin{example}
  Let $\mathcal{S}$ be the representation category of the symmetry group $S_3$. Then $\mathcal{S}$ has 3 simple objects $\textbf{1},\sigma,\psi$ of quantum dimensions 1,1,2 with nontrivial fusion rules $\sigma^2=\textbf{1}$, $\sigma\psi=\psi\sigma=\psi$ and $\psi^2=\textbf{1}+\sigma+\psi$.
  One can easily check that $H_\mathcal{S}=M_{3\times3}\oplus M_{3\times3}\oplus M_{5\times5}$. Hence $H_\mathcal{S}$ is not isomorphic to $\mathbb{C}[S_3]$.
\end{example}

\subsection{Representation theory of $H_{\mathcal{C}}$}

For a unitary fusion category $\mathcal{C}$, we can explicitly construct all simple representations for $H_{\mathcal{C}}$.
Let $V_i=span\{v^{ab}_i~|~i,a,b~\text{admissible}\}$. By graph, the basis vectors are presented by
\[
\begin{tikzpicture}[scale=0.5]
  \node (v) at (0,0) {$v^{ab}_i=$};
  \begin{scope}[xshift=2.5cm]
    \draw[line width=0.5mm] (0,0)--(1,0);
    \draw[line width=0.5mm] (0,0)--(-0.5,1);
    \draw[line width=0.5mm] (0,0)--(-0.5,-1);
    \node (a) at (-0.8,1.3) {$a$};
    \node (b) at (-0.8,-1.3) {$b$};
    \node (i) at (0.8,0.4) {$i$};
  \end{scope}
\end{tikzpicture}
\]
The $H_\mathcal{C}$-module structure on $V_i$ is given by
\[
\begin{tikzpicture}[scale=0.5]
\begin{scope}[xshift=-8cm]
  \draw[line width=0.5mm] (-1,0)--(1,0);
  \draw[line width=0.5mm] (-1,1)--(-1,-1);
  \draw[line width=0.5mm] (1,1)--(1,-1);
  \node (a) at (-1.3,1.3) {$a$};
  \node (b) at (-1.3,-1.3) {$b$};
  \node (c) at (1.3,1.3) {$c$};
  \node (d) at (1.3,-1.3) {$d$};
  \node (i) at (0,0.5) {$j$};
\end{scope}
  \begin{scope}[xshift=-4.2cm]
    \node (<) at (-1.7,0) {$\Bigg($};
    \node (>) at (1.7,0) {$\Bigg)$};
    \draw[line width=0.5mm] (0,0)--(1,0);
    \draw[line width=0.5mm] (0,0)--(-0.5,1);
    \draw[line width=0.5mm] (0,0)--(-0.5,-1);
    \node (a) at (-0.8,1.3) {$p$};
    \node (b) at (-0.8,-1.3) {$q$};
    \node (i) at (0.8,0.4) {$i$};
  \end{scope}
  \node (v) at (0,0) {$=\frac{\delta_{i,j}\delta_{c,p}\delta_{d,q}}{\sqrt{d_i}}$};
  \begin{scope}[xshift=3.5cm]
    \draw[line width=0.5mm] (0,0)--(1,0);
    \draw[line width=0.5mm] (0,0)--(-0.5,1);
    \draw[line width=0.5mm] (0,0)--(-0.5,-1);
    \node (a) at (-0.8,1.3) {$a$};
    \node (b) at (-0.8,-1.3) {$b$};
    \node (i) at (0.8,0.4) {$i$};
  \end{scope}
\end{tikzpicture}
\]
It is easy to check that the subalgebra $H^i_\mathcal{C}=span\{e^{ab}_{i;cb}~|~a,b,c,d,i ~\text{simple objects}\}\cong End(V_i)$. By counting the dimensions and unitarity of $H_\mathcal{C}$,
we have $H_\mathcal{C}=\oplus_iH^i_\mathcal{C}\cong \oplus_iEnd(V_i)$ and $V_i$'s are all simple representations of $H_\mathcal{C}$.

\begin{prop}
  The target counital subalgebra of $H_\mathcal{C}$ is given by
  \[
  \begin{tikzpicture}[scale=0.5]
    \node (Ht) at (0,0) {$H_t(H_\mathcal{C})=span_a$};
  \begin{scope}[xshift=4.5cm]
  \node (sum) at (0,0) {$\Bigg\{\sum\limits_{k,b}\frac{\sqrt{d_k}}{\sqrt{d_a}}$};
  \node (>) at (5.7,0) {$\Bigg\}$};
  \begin{scope}[xshift=3.5cm]
  \draw[line width=0.5mm] (-1,0)--(1,0);
  \draw[line width=0.5mm] (-1,1)--(-1,-1);
  \draw[line width=0.5mm] (1,1)--(1,-1);
  \node (a) at (-1.3,1.3) {$a$};
  \node (b) at (-1.3,-1.3) {$b$};
  \node (c) at (1.3,1.3) {$a$};
  \node (d) at (1.3,-1.3) {$b$};
  \node (i) at (0,0.5) {$k$};
  \end{scope}
  \end{scope}
  \end{tikzpicture}
  \]
  As vector spaces, $dim_\mathbb{C}H_t(H_\mathcal{C})=dim_\mathbb{C}V_\textbf{1}$. So $H_t(H_\mathcal{C})\cong V_\textbf{1}$ is simple.
\end{prop}

\begin{proof}
  It follows a direct calculation using the definition of the target counital subalgebra.
\end{proof}

Let us compute the fusion rule for $V_i$'s. We first observe that the tensor product
$$V_i\otimes V_j=\Delta(\eta)\cdot(V_i\otimes_\mathbb{C}V_j)=span\{v^{ab}_i\otimes v^{bc}_j~|~a,b,c,i,j ~\text{admissible}\}$$
In deed, it is easy to check this by computing $\Delta(\eta)$. The following proposition gives the Clebsch-Gordan coefficients for $V_j\otimes V_k$ and its decomposition into the direct sum of simple representations.
\begin{prop}
  The subspace of $V_j\otimes V_k$ spanned by
  \[
  \begin{tikzpicture}[scale=0.5]
  \node (v) at (0,0) {$u^{ab}_i=\sum\limits_{j,k}\sum\limits_p F^{ajk}_{b,ip}$};
  \begin{scope}[xshift=6cm,yshift=0.3cm]
  \begin{scope}[xshift=-1.7cm]
    \draw[line width=0.5mm] (0,0)--(1,0);
    \draw[line width=0.5mm] (0,0)--(-0.5,1);
    \draw[line width=0.5mm] (0,0)--(-0.5,-1);
    \node (a) at (-0.8,1.3) {$a$};
    \node (b) at (-0.8,-1.3) {$p$};
    \node (i) at (0.8,0.4) {$j$};
  \end{scope}
  \node (tensor) at (0,0) {$\otimes$};
  \begin{scope}[xshift=1.7cm]
    \draw[line width=0.5mm] (0,0)--(1,0);
    \draw[line width=0.5mm] (0,0)--(-0.5,1);
    \draw[line width=0.5mm] (0,0)--(-0.5,-1);
    \node (a) at (-0.8,1.3) {$p$};
    \node (b) at (-0.8,-1.3) {$b$};
    \node (i) at (0.8,0.4) {$k$};
  \end{scope}
  \end{scope}
  \end{tikzpicture}
  \]
  where the first summation runs over all $j,k$ such that $i,j,k$ are admissible, is isomorphic to $V_i$. Therefore, by counting the dimensions, we have
  $$V_j\otimes V_k\cong\bigoplus\limits_i V_i, ~~~\text{for} ~i,j,k ~\text{admissible}.$$
\end{prop}
\begin{proof}
We first calculate the coproduct using $F$-symbols. 
\[
  \begin{tikzpicture}[scale=0.5]
    \begin{scope}[xshift=-6.5cm]             
    \node (<) at (-4.3,0) {$\Bigg\langle$};
    \node (>) at (4.3,0) {$\Bigg\rangle$};
    \begin{scope}[xshift=-2cm]
  \draw[line width=0.5mm] (-1,0)--(1,0);
  \draw[line width=0.5mm] (-1,1)--(-1,-1);
  \draw[line width=0.5mm] (1,1)--(1,-1);
  \node (a) at (-1.3,1.3) {$a$};
  \node (b) at (-1.3,-1.3) {$b$};
  \node (c) at (1.3,1.3) {$c$};
  \node (d) at (1.3,-1.3) {$d$};
  \node (i) at (0,0.5) {$i$};
  \node (hat) at (0,1.9) {$\wedge$};
    \end{scope}
    \node (comma) at (0,-0.8) {$,$};
    \begin{scope}[xshift=2cm]
  \draw[line width=0.5mm] (-1,0.5)--(1,0.5);
  \draw[line width=0.5mm] (-1,-0.5)--(1,-0.5);
  \draw[line width=0.5mm] (-1,1.5)--(-1,-1.5);
  \draw[line width=0.5mm] (1,1.5)--(1,-1.5);
  \node (a) at (-1.3,1.8) {$a$};
  \node (b) at (-1.3,-1.8) {$b$};
  \node (c) at (1.3,1.8) {$c$};
  \node (d) at (1.3,-1.8) {$d$};
  \node (p) at (-1.3,0) {$p$};
  \node (q) at (1.3,0) {$q$};
  \node (j) at (0,1) {$j$};
  \node (k) at (0,-1) {$k$};
    \end{scope}
  \end{scope}
  \node (=) at (0,0) {$=\frac{\sqrt{d_i}}{\sqrt{d_jd_k}}$};
  \begin{scope}[xshift=6.3cm]             
  \node (<) at (-4.3,0) {$\Bigg\langle$};
    \node (>) at (5.5,0) {$\Bigg\rangle$};
    \begin{scope}[xshift=-2cm]
  \draw[line width=0.5mm] (-1,0)--(1,0);
  \draw[line width=0.5mm] (-1,1)--(-1,-1);
  \draw[line width=0.5mm] (1,1)--(1,-1);
  \node (a) at (-1.3,1.3) {$a$};
  \node (b) at (-1.3,-1.3) {$b$};
  \node (c) at (1.3,1.3) {$c$};
  \node (d) at (1.3,-1.3) {$d$};
  \node (i) at (0,0.5) {$i$};
  \node (hat) at (0,1.9) {$\wedge$};
    \end{scope}
    \node (comma) at (0,-0.8) {$,$};
    \begin{scope}[xshift=2.6cm]
  \draw[line width=0.5mm] (-1.5,0.7)--(-0.5,0);
  \draw[line width=0.5mm] (-1.5,-0.7)--(-0.5,0);
  \draw[line width=0.5mm] (1.5,-0.7)--(0.5,0);
  \draw[line width=0.5mm] (1.5,0.7)--(0.5,0);
  \draw[line width=0.5mm] (-1.5,1.5)--(-1.5,-1.5);
  \draw[line width=0.5mm] (1.5,1.5)--(1.5,-1.5);
  \draw[line width=0.5mm] (0.5,0)--(-0.5,0);
  \node (a) at (-1.8,1.8) {$a$};
  \node (b) at (-1.8,-1.8) {$b$};
  \node (c) at (1.8,1.8) {$c$};
  \node (d) at (1.8,-1.8) {$d$};
  \node (p) at (-1.8,0) {$p$};
  \node (q) at (1.8,0) {$q$};
  \node (j) at (0.8,0.8) {$j$};
  \node (j) at (-0.8,0.8) {$j$};
  \node (k) at (0.8,-0.8) {$k$};
  \node (k) at (-0.8,-0.8) {$k$};
  \node (i) at (0,0.5) {$i$};
    \end{scope}
  \end{scope}
  \end{tikzpicture}
  \]
  After removing the bubbles by two $F$-moves, we obtain a formula for coproduct in terms of $F$-symbols.
  \begin{equation}\label{eqn:Fcomultiplication}
  \begin{split}
  \begin{tikzpicture}[scale=0.5]
    \begin{scope}[xshift=-6.2cm]      
    \node (delta) at (-3,0) {$\Delta$};
    \node (<) at (-2.2,0) {$\Bigg($};
    \node (>) at (2.2,0) {$\Bigg)$};
    \draw[line width=0.5mm] (-1,0)--(1,0);
    \draw[line width=0.5mm] (-1,1)--(-1,-1);
    \draw[line width=0.5mm] (1,1)--(1,-1);
    \node (a) at (-1.3,1.3) {$a$};
    \node (b) at (-1.3,-1.3) {$b$};
    \node (c) at (1.3,1.3) {$c$};
    \node (d) at (1.3,-1.3) {$d$};
    \node (i) at (0,0.5) {$i$};
    \end{scope}
    \node (=) at (0,-0.4) {$=\sum\limits_{\substack{j,k\\p,q}}\frac{\sqrt{d_jd_k}}{\sqrt{d_i}}F^{ajk}_{b,ip}F^{dkj}_{c,iq}$};
    \begin{scope}[xshift=7.6cm]  
      \begin{scope}[xshift=-2.2cm]
    \draw[line width=0.5mm] (-1,0)--(1,0);
    \draw[line width=0.5mm] (-1,1)--(-1,-1);
    \draw[line width=0.5mm] (1,1)--(1,-1);
    \node (a) at (-1.3,1.3) {$a$};
    \node (b) at (-1.3,-1.3) {$p$};
    \node (c) at (1.3,1.3) {$c$};
    \node (d) at (1.3,-1.3) {$q$};
    \node (i) at (0,0.5) {$j$};
      \end{scope}
      \node (tensor) at (0,0) {$\otimes$};
      \begin{scope}[xshift=2.2cm]
    \draw[line width=0.5mm] (-1,0)--(1,0);
    \draw[line width=0.5mm] (-1,1)--(-1,-1);
    \draw[line width=0.5mm] (1,1)--(1,-1);
    \node (a) at (-1.3,1.3) {$p$};
    \node (b) at (-1.3,-1.3) {$b$};
    \node (c) at (1.3,1.3) {$q$};
    \node (d) at (1.3,-1.3) {$d$};
    \node (i) at (0,0.5) {$k$};
      \end{scope}
    \end{scope}
  \end{tikzpicture}
  \end{split}
  \end{equation}
  \noindent
  where $j,k$ must be admissible with $i$.
Using \eqref{eqn:Fcomultiplication}, we can end up the proof by checking that $e^{ab}_{i;cd}(u^{pq}_j)=\frac{\delta_{i,j}\delta_{c,p}\delta_{d,q}}{\sqrt{d_i}}u^{ab}_i$.
\[
  \begin{tikzpicture}[scale=0.5]
  \begin{scope}[xshift=-16.2cm]      
  \node (delta) at (-2.6,0) {$\Delta$};
  \node (<) at (-2,0) {$\Bigg($};
  \node (>) at (2,0) {$\Bigg)$};
  \draw[line width=0.5mm] (-1,0)--(1,0);
  \draw[line width=0.5mm] (-1,1)--(-1,-1);
  \draw[line width=0.5mm] (1,1)--(1,-1);
  \node (a) at (-1.3,1.3) {$a$};
  \node (b) at (-1.3,-1.3) {$b$};
  \node (c) at (1.3,1.3) {$a$};
  \node (d) at (1.3,-1.3) {$b$};
  \node (i) at (0,0.5) {$i$};
  \end{scope}
  \begin{scope}[xshift=-7.5cm]
  \node (<) at (-4.7,0) {$\Bigg(\sum\limits_{j,k,p}F^{ajk}_{b,ip}$};
  \node (>) at (3.1,0) {$\Bigg)$};
  \begin{scope}[xshift=-1.7cm]
    \draw[line width=0.5mm] (0,0)--(1,0);
    \draw[line width=0.5mm] (0,0)--(-0.5,1);
    \draw[line width=0.5mm] (0,0)--(-0.5,-1);
    \node (a) at (-0.8,1.3) {$a$};
    \node (b) at (-0.8,-1.3) {$p$};
    \node (i) at (0.8,0.4) {$j$};
  \end{scope}
  \node (tensor) at (0,0) {$\otimes$};
  \begin{scope}[xshift=1.5cm]
    \draw[line width=0.5mm] (0,0)--(1,0);
    \draw[line width=0.5mm] (0,0)--(-0.5,1);
    \draw[line width=0.5mm] (0,0)--(-0.5,-1);
    \node (a) at (-0.8,1.3) {$p$};
    \node (b) at (-0.8,-1.3) {$b$};
    \node (i) at (0.8,0.4) {$k$};
  \end{scope}
  \end{scope}
  \node (=) at (0,-0.4) {$=\sum\limits_{\substack{j,k\\q}}\frac{1}{\sqrt{d_i}}F^{ajk}_{b,iq}F^{bkj}_{a,ip}F^{ajk}_{b,ip}$};
  \begin{scope}[xshift=6.8cm]
  \begin{scope}[xshift=-1.7cm]
    \draw[line width=0.5mm] (0,0)--(1,0);
    \draw[line width=0.5mm] (0,0)--(-0.5,1);
    \draw[line width=0.5mm] (0,0)--(-0.5,-1);
    \node (a) at (-0.8,1.3) {$a$};
    \node (b) at (-0.8,-1.3) {$q$};
    \node (i) at (0.8,0.4) {$j$};
  \end{scope}
  \node (tensor) at (0,0) {$\otimes$};
  \begin{scope}[xshift=1.7cm]
    \draw[line width=0.5mm] (0,0)--(1,0);
    \draw[line width=0.5mm] (0,0)--(-0.5,1);
    \draw[line width=0.5mm] (0,0)--(-0.5,-1);
    \node (a) at (-0.8,1.3) {$q$};
    \node (b) at (-0.8,-1.3) {$b$};
    \node (i) at (0.8,0.4) {$k$};
  \end{scope}
  \end{scope}
  \end{tikzpicture}
  \]
  Note that $F^{ajk}_b(F^{bkj}_a)^\dag=I$. We have $e^{ab}_{i;ab}(u^{ab}_i)=\frac{1}{\sqrt{d_i}}u^{ab}_i$. The other cases follow by similar calculation.
\end{proof}

\begin{remark}
  If $\mathcal{C}$ is not multiplicity free, one can show in a similar way that the fusion rule for $H_\mathcal{C}$-modules is as the same as $\mathcal{C}$'s, i.e.,
  $$V_j\otimes V_k\cong\bigoplus\limits_i N^{i}_{jk}V_i, ~~~\text{for} ~i,j,k ~\text{admissible}.$$
\end{remark}

Note that the linear map given by $v^{ab}_i\mapsto u^{ab}_i$ is a basis vector of $Hom(V_i,V_j\otimes V_k)$. This gives rise to the $6j$-symbols for the representation category of $H_\mathcal{C}$. Explicitly, we have the followings bases for $Hom(V_d,(V_a\otimes V_b)\otimes V_c)$ and $Hom(V_d,V_a\otimes (V_b\otimes V_c))$:

\[
\begin{tikzpicture}[scale=0.5]    
  \begin{scope}[xshift=-2.5cm]
    \draw[line width=0.5mm] (0,0)--(-2,2);
    \draw[line width=0.5mm] (0,0)--(2,2);
    \draw[line width=0.5mm] (-1,1)--(0,2);
    \draw[line width=0.5mm] (0,0)--(0,-1);
    \node (a) at (-2,2.4) {$V_a$};
    \node (b) at (0,2.5) {$V_b$};
    \node (c) at (2,2.4) {$V_c$};
    \node (d) at (0,-1.4) {$V_d$};
    \node (m) at (-1,0.3) {$V_m$};
  \end{scope}
  \node (=) at (0,0.3) {$:$};
  \begin{scope}[xshift=1.5cm,yshift=0.3cm]
    \draw[line width=0.5mm] (0,0)--(1,0);
    \draw[line width=0.5mm] (0,0)--(-0.5,1);
    \draw[line width=0.5mm] (0,0)--(-0.5,-1);
    \node (a) at (-0.8,1.3) {$i$};
    \node (b) at (-0.8,-1.3) {$j$};
    \node (i) at (0.8,0.4) {$d$};
  \end{scope}
  \node (->) at (5.8,0.3) {$\mapsto\sum\limits_{p,q}F^{imc}_{j,dp}F^{iab}_{p,mq}$};
  \begin{scope}[xshift=11.3cm,yshift=0.3cm]
  \begin{scope}[xshift=-1.5cm]
    \draw[line width=0.5mm] (0,0)--(1,0);
    \draw[line width=0.5mm] (0,0)--(-0.5,1);
    \draw[line width=0.5mm] (0,0)--(-0.5,-1);
    \node (a) at (-0.8,1.3) {$i$};
    \node (b) at (-0.8,-1.3) {$q$};
    \node (i) at (0.8,0.4) {$a$};
  \end{scope}
  \node (tensor) at (0,0) {$\otimes$};
  \begin{scope}[xshift=1.5cm]
    \draw[line width=0.5mm] (0,0)--(1,0);
    \draw[line width=0.5mm] (0,0)--(-0.5,1);
    \draw[line width=0.5mm] (0,0)--(-0.5,-1);
    \node (a) at (-0.8,1.3) {$q$};
    \node (b) at (-0.8,-1.3) {$p$};
    \node (i) at (0.8,0.4) {$b$};
  \end{scope}
    \node (tensor) at (3,0) {$\otimes$};
  \begin{scope}[xshift=4.5cm]
    \draw[line width=0.5mm] (0,0)--(1,0);
    \draw[line width=0.5mm] (0,0)--(-0.5,1);
    \draw[line width=0.5mm] (0,0)--(-0.5,-1);
    \node (a) at (-0.8,1.3) {$p$};
    \node (b) at (-0.8,-1.3) {$j$};
    \node (i) at (0.8,0.4) {$c$};
  \end{scope}
  \end{scope}
\end{tikzpicture}
\]

\[
\begin{tikzpicture}[scale=0.5]    
  \begin{scope}[xshift=-2.5cm]
    \draw[line width=0.5mm] (0,0)--(-2,2);
    \draw[line width=0.5mm] (0,0)--(2,2);
    \draw[line width=0.5mm] (1,1)--(0,2);
    \draw[line width=0.5mm] (0,0)--(0,-1);
    \node (a) at (-2,2.4) {$V_a$};
    \node (b) at (0,2.5) {$V_b$};
    \node (c) at (2,2.4) {$V_c$};
    \node (d) at (0,-1.4) {$V_d$};
    \node (n) at (1.2,0.3) {$V_n$};
  \end{scope}
  \node (=) at (0,0.3) {$:$};
  \begin{scope}[xshift=1.5cm,yshift=0.3cm]
    \draw[line width=0.5mm] (0,0)--(1,0);
    \draw[line width=0.5mm] (0,0)--(-0.5,1);
    \draw[line width=0.5mm] (0,0)--(-0.5,-1);
    \node (a) at (-0.8,1.3) {$i$};
    \node (b) at (-0.8,-1.3) {$j$};
    \node (i) at (0.8,0.4) {$d$};
  \end{scope}
  \node (->) at (5.8,0.3) {$\mapsto\sum\limits_{p,q}F^{ian}_{j,dq}F^{qbc}_{j,np}$};
  \begin{scope}[xshift=11.3cm,yshift=0.3cm]
  \begin{scope}[xshift=-1.5cm]
    \draw[line width=0.5mm] (0,0)--(1,0);
    \draw[line width=0.5mm] (0,0)--(-0.5,1);
    \draw[line width=0.5mm] (0,0)--(-0.5,-1);
    \node (a) at (-0.8,1.3) {$i$};
    \node (b) at (-0.8,-1.3) {$q$};
    \node (i) at (0.8,0.4) {$a$};
  \end{scope}
  \node (tensor) at (0,0) {$\otimes$};
  \begin{scope}[xshift=1.5cm]
    \draw[line width=0.5mm] (0,0)--(1,0);
    \draw[line width=0.5mm] (0,0)--(-0.5,1);
    \draw[line width=0.5mm] (0,0)--(-0.5,-1);
    \node (a) at (-0.8,1.3) {$q$};
    \node (b) at (-0.8,-1.3) {$p$};
    \node (i) at (0.8,0.4) {$b$};
  \end{scope}
    \node (tensor) at (3,0) {$\otimes$};
  \begin{scope}[xshift=4.5cm]
    \draw[line width=0.5mm] (0,0)--(1,0);
    \draw[line width=0.5mm] (0,0)--(-0.5,1);
    \draw[line width=0.5mm] (0,0)--(-0.5,-1);
    \node (a) at (-0.8,1.3) {$p$};
    \node (b) at (-0.8,-1.3) {$j$};
    \node (i) at (0.8,0.4) {$c$};
  \end{scope}
  \end{scope}
\end{tikzpicture}
\]

\noindent
By the pentagon equation $\sum\limits_mF^{abc}_{d,nm}F^{imc}_{j,dp}F^{iab}_{p,mq}=F^{ian}_{j,dq}F^{qbc}_{j,np}$ and the unitarity of $F^{abc}_{d}$, we get the $F$-moves for the representation category of $H_\mathcal{C}$.
\[
\begin{tikzpicture}[scale=0.5]    
  \begin{scope}[xshift=-4.8cm]
    \draw[line width=0.5mm] (0,0)--(-2,2);
    \draw[line width=0.5mm] (0,0)--(2,2);
    \draw[line width=0.5mm] (-1,1)--(0,2);
    \draw[line width=0.5mm] (0,0)--(0,-1);
    \node (a) at (-2,2.4) {$V_a$};
    \node (b) at (0,2.5) {$V_b$};
    \node (c) at (2,2.4) {$V_c$};
    \node (d) at (0,-1.4) {$V_d$};
    \node (m) at (-1,0.3) {$V_m$};
  \end{scope}
  \node (=) at (0,0.3) {$=\sum\limits_n\overline{F^{abc}_{d,nm}}$};
  \begin{scope}[xshift=4.8cm]
    \draw[line width=0.5mm] (0,0)--(-2,2);
    \draw[line width=0.5mm] (0,0)--(2,2);
    \draw[line width=0.5mm] (1,1)--(0,2);
    \draw[line width=0.5mm] (0,0)--(0,-1);
    \node (a) at (-2,2.4) {$V_a$};
    \node (b) at (0,2.5) {$V_b$};
    \node (c) at (2,2.4) {$V_c$};
    \node (d) at (0,-1.4) {$V_d$};
    \node (n) at (1.2,0.3) {$V_n$};
  \end{scope}
\end{tikzpicture}
\]

It is clear that the above argument can be applied to the general unitary fusion category, not relying on the assumption of being self-dual and multiplicity free and trivial Frobenius-Schur indicators. As a consequence, we have the following theorem.

\begin{theorem}
  Given a unitary fusion category $\mathcal{C}$ with real $F$-matrices, the representation category of $H_\mathcal{C}$ is equivalent to $\mathcal{C}$.
\end{theorem}

This is a special case of a general theorem in \cite{ENO}.

\vspace{.1in}
{\bf Theorem.} \cite{ENO}
{\it 
  Any fusion category is equivalent to the category of finite dimensional representations of some semisimple quantum groupoid.
}

\section{Kitaev model based on $H_\mathcal{C}$}

In the following we construct a frustration-free Hamiltonian by some suitable choice of cocommutative elements in $H_\mathcal{C}$ and $H_\mathcal{C}^*$.
First, consider the following element $\Lambda$ in $H_\mathcal{C}$.
\[
\begin{tikzpicture}[scale=0.5]
  \node (Lambda) at (0,-0.4) {$\Lambda=\frac{1}{D^2}\sum\limits_{a,b}d_ad_b$};
  \begin{scope}[xshift=4.4cm]
  \draw[line width=0.5mm] (-1,1)--(-1,-1);
  \draw[line width=0.5mm] (1,1)--(1,-1);
  \draw[line width=0.5mm] (-1,0)--(1,0);
  \node (a) at (-1.3,1.3) {$a$};
  \node (b) at (-1.3,-1.3) {$a$};
  \node (c) at (1.3,1.3) {$b$};
  \node (d) at (1.3,-1.3) {$b$};
  \node (i) at (0,0.5) {$\textbf{1}$};
  \end{scope}
\end{tikzpicture}
\]

\begin{lemma}
  $\Lambda$ is cocommutative, $S$-invariant and $*$-invariant, which means that $\Delta^{op}(\Lambda)=\Delta(\Lambda)$, $S(\Lambda)=\Lambda$ and $\Lambda^*=\Lambda$.
  It is an idempotent. Moreover,
  \[
\begin{tikzpicture}[scale=0.5]
  \node (Delta) at (-2,-0.8) {$\Delta^2(\Lambda)=\frac{1}{D^2}\sum\limits_{\substack{a,c,p\\b,d,q\\i,j,k}}\sqrt{d_id_jd_k}$};
  \begin{scope}[xshift=5.5cm]
  \node (<) at (-2.6,0.2) {$\big\langle$};
  \node (>) at (2.6,0.2) {$\big\rangle$};
  \draw[line width=0.5mm] (0,1)--(0,2);
  \draw[line width=0.5mm] (0.86,-0.5)--(1.73,-1);
  \draw[line width=0.5mm] (-0.86,-0.5)--(-1.73,-1);
  \draw[line width=0.5mm] (1,0) arc(0:360:1);
  \node (a) at (-0.1,2.4) {$i$};
  \node (b) at (-2,-1.2) {$k$};
  \node (c) at (2,-1.2) {$j$};
  \node (m) at (1.6,0.2) {$c$};
  \node (n) at (-1.5,0.2) {$a$};
  \node (k) at (0,-1.5) {$p$};
  \end{scope}
  \begin{scope}[xshift=11.2cm]
  \node (<) at (-2.6,0.2) {$\big\langle$};
  \node (>) at (2.6,0.2) {$\big\rangle$};
  \draw[line width=0.5mm] (0,1)--(0,2);
  \draw[line width=0.5mm] (0.86,-0.5)--(1.73,-1);
  \draw[line width=0.5mm] (-0.86,-0.5)--(-1.73,-1);
  \draw[line width=0.5mm] (1,0) arc(0:360:1);
  \node (a) at (-0.1,2.4) {$i$};
  \node (b) at (-2,-1.2) {$j$};
  \node (c) at (2,-1.2) {$k$};
  \node (m) at (1.6,0.2) {$b$};
  \node (n) at (-1.5,0.2) {$d$};
  \node (k) at (0,-1.5) {$q$};
  \end{scope}
  \begin{scope}[xshift=18.1cm]  
    \begin{scope}[xshift=-2.2cm]
  \draw[line width=0.5mm] (-1,0)--(1,0);
  \draw[line width=0.5mm] (-1,1)--(-1,-1);
  \draw[line width=0.5mm] (1,1)--(1,-1);
  \node (a) at (-1.3,1.3) {$a$};
  \node (b) at (-1.3,-1.3) {$c$};
  \node (c) at (1.3,1.3) {$b$};
  \node (d) at (1.3,-1.3) {$d$};
  \node (i) at (0,0.5) {$i$};
    \end{scope}
    \node (tensor) at (0,0) {$\otimes$};
    \begin{scope}[xshift=2.2cm]
  \draw[line width=0.5mm] (-1,0)--(1,0);
  \draw[line width=0.5mm] (-1,1)--(-1,-1);
  \draw[line width=0.5mm] (1,1)--(1,-1);
  \node (a) at (-1.3,1.3) {$c$};
  \node (b) at (-1.3,-1.3) {$p$};
  \node (c) at (1.3,1.3) {$d$};
  \node (d) at (1.3,-1.3) {$q$};
  \node (i) at (0,0.5) {$j$};
    \end{scope}
      \node (tensor) at (4.4,0) {$\otimes$};
    \begin{scope}[xshift=6.6cm]
  \draw[line width=0.5mm] (-1,0)--(1,0);
  \draw[line width=0.5mm] (-1,1)--(-1,-1);
  \draw[line width=0.5mm] (1,1)--(1,-1);
  \node (a) at (-1.3,1.3) {$p$};
  \node (b) at (-1.3,-1.3) {$a$};
  \node (c) at (1.3,1.3) {$q$};
  \node (d) at (1.3,-1.3) {$b$};
  \node (i) at (0,0.5) {$k$};
    \end{scope}
  \end{scope}
\end{tikzpicture}
\]
where $i,j,k$ inside the summation must be admissible.
\end{lemma}
\begin{proof}
  First, $S$-invariance is obvious. Next by direct calculation, we have $\Lambda^2=\Lambda$ and
 \[
\begin{tikzpicture}[scale=0.5]
  \node (Delta) at (0,-0.5) {$\Delta(\Lambda)=\frac{1}{D^2}\sum\limits_{a,b,c,d,i}\sqrt{d_ad_bd_cd_d}$};
  \begin{scope}[xshift=9cm]  
    \begin{scope}[xshift=-2.2cm]
  \draw[line width=0.5mm] (-1,0)--(1,0);
  \draw[line width=0.5mm] (-1,1)--(-1,-1);
  \draw[line width=0.5mm] (1,1)--(1,-1);
  \node (a) at (-1.3,1.3) {$a$};
  \node (b) at (-1.3,-1.3) {$c$};
  \node (c) at (1.3,1.3) {$b$};
  \node (d) at (1.3,-1.3) {$d$};
  \node (i) at (0,0.5) {$i$};
    \end{scope}
    \node (tensor) at (0,0) {$\otimes$};
    \begin{scope}[xshift=2.2cm]
  \draw[line width=0.5mm] (-1,0)--(1,0);
  \draw[line width=0.5mm] (-1,1)--(-1,-1);
  \draw[line width=0.5mm] (1,1)--(1,-1);
  \node (a) at (-1.3,1.3) {$c$};
  \node (b) at (-1.3,-1.3) {$a$};
  \node (c) at (1.3,1.3) {$d$};
  \node (d) at (1.3,-1.3) {$b$};
  \node (i) at (0,0.5) {$i$};
    \end{scope}
  \end{scope}
\end{tikzpicture}
\]
It implies that $\Delta^{op}(\Lambda)=\Delta(\Lambda)$ by symmetry. To calculate $\Delta^2(\Lambda)$, we need to evaluate the coefficient
\[
  \begin{tikzpicture}[scale=0.5]
    \begin{scope}[xshift=-6cm]             
    \node (<) at (-4.3,0) {$\Bigg\langle$};
    \node (>) at (4.3,0) {$\Bigg\rangle$};
    \begin{scope}[xshift=-2cm]
  \draw[line width=0.5mm] (-1,1)--(-1,-1);
  \draw[line width=0.5mm] (1,1)--(1,-1);
  \draw[line width=0.5mm] (-1,0)--(1,0);
  \node (a) at (-1.3,1.3) {$$};
  \node (b) at (-1.3,-1.3) {$a$};
  \node (c) at (1.3,1.3) {$$};
  \node (d) at (1.3,-1.3) {$b$};
  \node (i) at (0,0.5) {$\textbf{1}$};
  \node (hat) at (0,1.9) {$\wedge$};
    \end{scope}
    \node (comma) at (0,-0.8) {$,$};
    \begin{scope}[xshift=2cm]
  \draw[line width=0.5mm] (-1,1)--(1,1);
  \draw[line width=0.5mm] (-1,0)--(1,0);
  \draw[line width=0.5mm] (-1,-1)--(1,-1);
  \draw[line width=0.5mm] (-1,2)--(-1,-2);
  \draw[line width=0.5mm] (1,2)--(1,-2);
  \node (a) at (-1.3,2.3) {$a$};
  \node (b) at (-1.3,-2.3) {$a$};
  \node (c) at (1.3,2.3) {$b$};
  \node (d) at (1.3,-2.3) {$b$};
  \node (p) at (-1.3,0.5) {$c$};
  \node (q) at (1.3,0.5) {$d$};
  \node (p) at (-1.3,-0.6) {$p$};
  \node (q) at (1.3,-0.6) {$q$};
  \node (i) at (0,1.4) {$i$};
  \node (j) at (0,0.4) {$j$};
  \node (k) at (0,-0.6) {$k$};
    \end{scope}
  \end{scope}
  \node (=) at (0,0) {$=\frac{1}{d_ad_b}$};
  \begin{scope}[xshift=4cm,scale=1.2]
  \draw[line width=0.5mm] (0,1)--(0,2);
  \draw[line width=0.5mm] (0.86,-0.5)--(1.73,-1);
  \draw[line width=0.5mm] (-0.86,-0.5)--(-1.73,-1);
  \draw[line width=0.5mm] (1,0) arc(0:360:1);
  \draw[line width=0.5mm] (2,0) arc(0:360:2);
  \node (a) at (1.8,1.6) {$a$};
  \node (b) at (-1.8,1.6) {$c$};
  \node (c) at (0,-2.5) {$p$};
  \node (m) at (-0.4,1.4) {$i$};
  \node (n) at (-1.5,-0.5) {$j$};
  \node (k) at (1.5,-0.5) {$k$};
  \node (p) at (1.2,0.7) {$b$};
  \node (q) at (-1.2,0.7) {$d$};
  \node (r) at (0,-0.6) {$q$};
  \end{scope}
  \end{tikzpicture}
  \]
  By the same calculation as Lemma \ref{prop:comultiplication}, we obtain the formula.
\end{proof}

Similarly, in $H_\mathcal{C}^*$, we have
\[
\begin{tikzpicture}[scale=0.5]
  \node (lambda) at (0,-0.4) {$\lambda=\frac{1}{D^2}\sum\limits_{a,b,\mu}\sqrt{d_\mu}$};
  \begin{scope}[xshift=4.8cm]
  \draw[line width=0.5mm] (-1,1)--(-1,-1);
  \draw[line width=0.5mm] (1,1)--(1,-1);
  \draw[line width=0.5mm] (1,0)--(-1,0);
  \draw[line width=0.5mm] (-1,1)--(-1,-1);
  \draw[line width=0.5mm] (1,1)--(1,-1);
  \node (a) at (-1.3,1.3) {$a$};
  \node (b) at (-1.3,-1.3) {$a$};
  \node (c) at (1.3,1.3) {$b$};
  \node (d) at (1.3,-1.3) {$b$};
  \node (i) at (0,0.5) {$\mu$};
  \node (hat) at (0,1.9) {$\wedge$};
  \end{scope}
\end{tikzpicture}
\]
A straightforward calculation leads to the following properties of $\lambda$.
\begin{lemma}
  $\lambda$ is an idempotent in $H_\mathcal{C}^*$ and $*$-invariant.
  It induces an $S$-invariant trace on $H_\mathcal{C}^*$. That is
  $\lambda(yx)=\lambda(xy)=\lambda(x)\lambda(y)$, $\lambda(S(x))=\lambda(x)$ and $\lambda(x^*)=\lambda(x)$ for $x,y\in H_\mathcal{C}$.
\end{lemma}

Now we define the vertex operator $A^K_\mathbf{v}:=A_\Lambda(\mathbf{v},\mathbf{p})$ and plaquette operator $B^K_\mathbf{p}:=B_\lambda(\mathbf{v},\mathbf{p})$ for $H_\mathcal{H}$. Both of them are Hermitian. This can be verified by checking that
$(L^h_{\pm})^\dag=L^{h^*}_{\pm}$ and $(T^\phi_{\pm})^\dag=T^{\phi^*}_{\pm}$ for $h\in H_\mathcal{C}$ and $\phi\in H_\mathcal{C}^*$.
For $L^h_{\pm}$,
\begin{eqnarray*}
  (x,L^h_{+}(y))&=&\langle\chi,x^*hy\rangle=\langle\chi,(h^*x)^*y\rangle=(L^{h^*}_{+}x,y)\\
  (x,L^h_{-}(y))&=&\langle\chi,x^*yS(h)\rangle=\langle\chi,S^{-1}(h)x^*y\rangle=\langle\chi,(xS(h))^*y\rangle=(L^{h^*}_{-}x,y)
\end{eqnarray*}
Here we have used that $\chi(xy)=\chi(yS^2(x))$ and $S(S(x^*)^*)=x$ for $x,y\in H_\mathcal{C}$.
Similarly, one can find the adjoints for $T^\phi_{\pm}$. Hence both $A^K_\mathbf{v}$ and $B^K_\mathbf{p}$ are Hermitian since $\Lambda$ and $\lambda$ are $*$-invariant. Applying Proposition \ref{prop:ABoperators}, we obtain a frustration-free Hamiltonian.

\begin{prop}
  Given a lattice $\Gamma$ on an oriented closed surface $\Sigma$ and a unitary fusion category $\mathcal{C}$,
  $$\mathcal{H}^K=-\sum\limits_\mathbf{v}A^K_\mathbf{v}-\sum\limits_\mathbf{p}B^K_\mathbf{p}$$
  is a frustration-free unitary Hamiltonian for the Kitaev model based on $H_\mathcal{C}$.
\end{prop}

Because the Hamiltonian is a sum of local commuting projectors, the ground state space $\mathcal{G}^{K}(\Sigma,\Gamma)$ consists of the vector $|\Psi\rangle\in\mathcal{L}^K$ such that $A^K_\mathbf{v}|\Psi\rangle=B^K_\mathbf{p}|\Psi\rangle =|\Psi\rangle$ for all $\mathbf{v}$ and $\mathbf{p}$. In order to describe the ground states of the Kitaev model, we first analyze the image of $A^K_\mathbf{v}$ in $\mathcal{L}^K$. For this, we calculate the action of $A^K_\mathbf{v}$ on a basis vector of $\mathcal{L}^K$.
\[

\]
where $i_1,i_2,i_3$ must be admissible otherwise the result equals zero.
Here the graph in the ket presents the labelled lattice around the vertex $\mathbf{v}$ and the dash lines are the edges of the oriented lattice.
Note that the action depends on the orientation of the edges.
This is the reason that the factor $\frac{\sqrt{d_c}}{\sqrt{d_p}}$ appears in the formula.
By this formula, $A^K_\mathbf{v}(\mathcal{L}^K)$ has a basis around vertex $\mathbf{v}$ given by
\[
%
\]
Note that this definition depends on the edge orientation at $\mathbf{v}$. For different pattern of edge orientation at $\mathbf{v}$, $A^K_\mathbf{v}$ may be different. One needs to define the basis for $A^K_\mathbf{v}(\mathcal{L}^K)$ accordingly.
\[
%
\]

Let $\mathcal{L}^K_0$ be the surviving subspace of $\mathcal{L}^K$ after the actions of all $A^K_{\mathbf{v}}$'s. A basis vector of $\mathcal{L}^K_0$ is
\[
%
\]
where the coefficient $C_{a,b,c}$ is given by
\[
%
\]
Note that we have omitted the factors associated to those vertices not on plaquette $\mathbf{p}$
and an overall factor $(\frac{1}{D^2})^N$ where $N$ is the number of vertices.
Because the plaquette operators $B_\mathbf{p}$'s are local linear operators, the dropped factors do not affect the discussion about the action of $B_\mathbf{p}$'s on $\mathcal{L}^K_0$ and the description of the ground states.

In the following, we work on the action of $B^K_{\mathbf{p}}$ on $\mathcal{L}^K_0$. As similar as Levin-Wen model, we write $B^K_{\mathbf{p}}$ as
$$B^K_{\mathbf{p}}=\sum\limits_\mu\frac{d_\mu}{D^2}B^K_{\lambda_\mu}(\mathbf{v},\mathbf{p})$$
where
\[

\]
One can also check that $\lambda_\mu$ induces an $S$-invariant trace on $H_\mathcal{C}$.
We simply denote $B^K_{\lambda_\mu}(\mathbf{v},\mathbf{p})$ by $B^\mu_{\mathbf{p}}$
 whose action on the basis of $\mathcal{L}^K_0$ is given in the following proposition.
\begin{prop}\label{prop:Kitaevplaqutte}
\begin{equation}\label{eqn:Kitaevplaqutte}
\begin{split}
%
\]
The last equality results form the pentagon identity for symmetric $6j$-symbols.
In fact, we have applied the pentagon identity at each vertex.
The last vector is in $\mathcal{L}^K_0$. Indeed, it is
\[
%
\]
which coincides with the right hand side of \eqref{eqn:Kitaevplaqutte} in Propostion \ref{prop:Kitaevplaqutte}.
\end{proof}
Note that the formula in proposition 10 is for one particular pattern of orientation.
If we reverse the orientation of an edge $e$, then the map $x_e\mapsto S(x_e)$
is compatible with the actions $L^h_{\pm}$ and $T^\phi_{\pm}$.
So all models with different patterns of orientation are equivariant and their ground states are independent of
the orientation of edges. For example, if $L^\Lambda_{-}(x_e)=x_e$, then $S^2(\Lambda)S(x_e)=S(x_e)$ and so $L^\Lambda_{+}(S(x_e))=S(x_e)$ for $\Lambda$ is $S$-invariant.

Let $\Theta:\mathcal{L}^K_0\rightarrow\mathcal{L}^{LW}_0$ given by
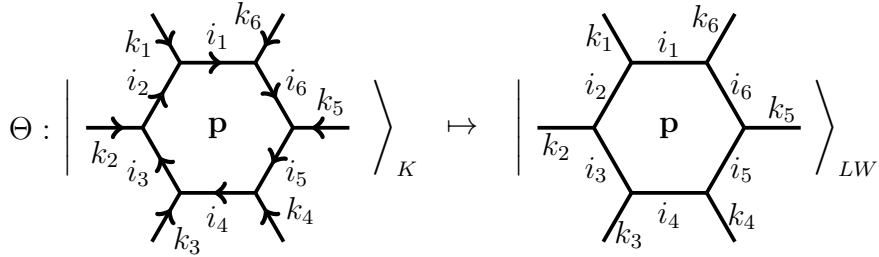
\[
\begin{tikzpicture}[scale=0.5]
  \node (theta) at (0,0) {$\Theta:$};
  \begin{scope}[xshift=5cm]
    \node (<) at (-4,0) {$\Bigg|$};
\node (>) at (4.7,0) {$\Bigg\rangle_{K}$};
\node (p) at (0,0) {$\mathbf{p}$};
  \begin{scope}[rotate=120]             
  \begin{scope}[xshift=2cm]
    \draw[line width=0.5mm] (0,0)--(1.5,0);
    \draw[<-,line width=0.7mm] (0.5,0)--(0.6,0);
    \node (k) at (1,0.7) {$k_1$};
    \begin{scope}[rotate=-120]
      \draw[line width=0.5mm] (0,0)--(2,0);
      \draw[->,line width=0.7mm] (1,0)--(1.1,0);
      \node (i) at (1,0.7) {$i_1$};
    \end{scope}
  \end{scope}
  \end{scope}
  \begin{scope}[rotate=180]             
  \begin{scope}[xshift=2cm]
    \draw[line width=0.5mm] (0,0)--(1.5,0);
    \draw[<-,line width=0.7mm] (0.5,0)--(0.6,0);
    \node (k) at (1,0.7) {$k_2$};
    \begin{scope}[rotate=-120]
      \draw[line width=0.5mm] (0,0)--(2,0);
      \draw[->,line width=0.7mm] (1,0)--(1.1,0);
      \node (i) at (1,0.7) {$i_2$};
    \end{scope}
  \end{scope}
  \end{scope}
  \begin{scope}[rotate=240]             
  \begin{scope}[xshift=2cm]
    \draw[line width=0.5mm] (0,0)--(1.5,0);
    \draw[<-,line width=0.7mm] (0.5,0)--(0.6,0);
    \node (k) at (1,0.8) {$k_3$};
    \begin{scope}[rotate=-120]
      \draw[line width=0.5mm] (0,0)--(2,0);
      \draw[->,line width=0.7mm] (1,0)--(1.1,0);
      \node (i) at (1,0.7) {$i_3$};
    \end{scope}
  \end{scope}
  \end{scope}
  \begin{scope}[rotate=300]             
  \begin{scope}[xshift=2cm]
    \draw[line width=0.5mm] (0,0)--(1.5,0);
    \draw[<-,line width=0.7mm] (0.5,0)--(0.6,0);
    \node (k) at (1,0.8) {$k_4$};
    \begin{scope}[rotate=-120]
      \draw[line width=0.5mm] (0,0)--(2,0);
      \draw[->,line width=0.7mm] (1,0)--(1.1,0);
      \node (i) at (1,0.7) {$i_4$};
    \end{scope}
  \end{scope}
  \end{scope}
  \begin{scope}[rotate=0]             
  \begin{scope}[xshift=2cm]
    \draw[line width=0.5mm] (0,0)--(1.5,0);
    \draw[<-,line width=0.7mm] (0.5,0)--(0.6,0);
    \node (k) at (1,0.7) {$k_5$};
    \begin{scope}[rotate=-120]
      \draw[line width=0.5mm] (0,0)--(2,0);
      \draw[->,line width=0.7mm] (1,0)--(1.1,0);
      \node (i) at (1,0.7) {$i_5$};
    \end{scope}
  \end{scope}
  \end{scope}
  \begin{scope}[rotate=60]             
  \begin{scope}[xshift=2cm]
    \draw[line width=0.5mm] (0,0)--(1.5,0);
    \draw[<-,line width=0.7mm] (0.5,0)--(0.6,0);
    \node (k) at (1,0.8) {$k_6$};
    \begin{scope}[rotate=-120]
      \draw[line width=0.5mm] (0,0)--(2,0);
      \draw[->,line width=0.7mm] (1,0)--(1.1,0);
      \node (i) at (1,0.7) {$i_6$};
    \end{scope}
  \end{scope}
  \end{scope}
  \end{scope}
\node (map) at (11.5,0) {$\mapsto$};
\begin{scope}[xshift=17cm]
  \node (<) at (-4,0) {$\Bigg|$};
\node (>) at (4.7,0) {$\Bigg\rangle_{LW}$};
\node (p) at (0,0) {$\mathbf{p}$};
  \begin{scope}[rotate=120]             
  \begin{scope}[xshift=2cm]
    \draw[line width=0.5mm] (0,0)--(1.5,0);
    \node (k) at (1,0.4) {$k_1$};
    \begin{scope}[rotate=-120]
      \draw[line width=0.5mm] (0,0)--(2,0);
      \node (i) at (1,0.5) {$i_1$};
    \end{scope}
  \end{scope}
  \end{scope}
  \begin{scope}[rotate=180]             
  \begin{scope}[xshift=2cm]
    \draw[line width=0.5mm] (0,0)--(1.5,0);
    \node (k) at (1,0.5) {$k_2$};
    \begin{scope}[rotate=-120]
      \draw[line width=0.5mm] (0,0)--(2,0);
      \node (i) at (1,0.5) {$i_2$};
    \end{scope}
  \end{scope}
  \end{scope}
  \begin{scope}[rotate=240]             
  \begin{scope}[xshift=2cm]
    \draw[line width=0.5mm] (0,0)--(1.5,0);
    \node (k) at (1,0.5) {$k_3$};
    \begin{scope}[rotate=-120]
      \draw[line width=0.5mm] (0,0)--(2,0);
      \node (i) at (1,0.5) {$i_3$};
    \end{scope}
  \end{scope}
  \end{scope}
  \begin{scope}[rotate=300]             
  \begin{scope}[xshift=2cm]
    \draw[line width=0.5mm] (0,0)--(1.5,0);
    \node (k) at (1,0.5) {$k_4$};
    \begin{scope}[rotate=-120]
      \draw[line width=0.5mm] (0,0)--(2,0);
      \node (i) at (1,0.6) {$i_4$};
    \end{scope}
  \end{scope}
  \end{scope}
  \begin{scope}[rotate=0]             
  \begin{scope}[xshift=2cm]
    \draw[line width=0.5mm] (0,0)--(1.5,0);
    \node (k) at (1,0.5) {$k_5$};
    \begin{scope}[rotate=-120]
      \draw[line width=0.5mm] (0,0)--(2,0);
      \node (i) at (1,0.5) {$i_5$};
    \end{scope}
  \end{scope}
  \end{scope}
  \begin{scope}[rotate=60]             
  \begin{scope}[xshift=2cm]
    \draw[line width=0.5mm] (0,0)--(1.5,0);
    \node (k) at (1,0.6) {$k_6$};
    \begin{scope}[rotate=-120]
      \draw[line width=0.5mm] (0,0)--(2,0);
      \node (i) at (1,0.5) {$i_6$};
    \end{scope}
  \end{scope}
  \end{scope}
\end{scope}
\end{tikzpicture}
\]
Then it is clear that $\Theta$ is bijective and the above proposition implies that
$\Theta\circ B^K_\mathbf{p}=B^{LW}_\mathbf{p}\circ\Theta$ for all plaquette.
Therefore, we know that the ground states of these two models are in 1-1 correspondence.

\begin{theorem}
Given a lattice $\Gamma$ on a closed oriented surface $\Sigma$ and a unitary fusion category $\mathcal{C}$ that is multiplicity free and whose simple objects are self-dual with trivial Frobenius-Schur indicators, the ground state space $\mathcal{G}^{K}(\Sigma,\Gamma)$ of the Kitaev model based on $H_\mathcal{C}$ is canonically isomorphic to the ground state space $\mathcal{G}^{LW}(\Sigma,\Gamma)$ of Levin-Wen Models based on $\mathcal{C}$.
\end{theorem}

Combining with the result in \cite{Kir}, one has $\mathcal{G}^{K}(\Sigma,\Gamma)$ is canonically isomorphic to the target space $Z_{TV}(\Sigma)$ of the TV-TQFT based on $\mathcal{C}$. This implies that the ground state does not depend on the choice of trivalent lattice and is a topological invariant of $\Sigma$.

The self-duality and multiplicity free assumptions for $\mathcal{C}$ are not essential obstacles to establish a general model. Their main use is to save indices and arrows when doing graphical calculus. The assumption of trivial Frobenius-Schur indicators for $\mathcal{C}$ is subtle. It is not needed to write down a frustration free Hamiltonian based on $H_\mathcal{C}$ for a general unitary fusion category $\mathcal{C}$. However, in order to show the new models have the same ground states as LW models, we need to impose the trivial FS indicators assumption to deal with the identities evolving $6j$ symbols with certain symmetry. It is expected to drop this assumption by more careful discussion in the future.

\end{document}